\documentclass[11pt]{article}
\usepackage{amsmath,amssymb,latexsym}
\usepackage{amsthm}
\usepackage{graphicx}
\usepackage{enumitem}
\usepackage[T1]{fontenc}
\usepackage{rotating,amssymb,amsmath,amsfonts,delarray}
\usepackage[ruled,vlined]{algorithm2e}

\usepackage{color}
\usepackage{amsfonts}
\usepackage{hyperref}
\usepackage{breakcites}
\usepackage{booktabs}
\usepackage{tikz}
\usepackage{pgfplots}
\usepackage{todonotes}
\pgfplotsset{compat=1.17}
\usepackage{diagbox}
\usetikzlibrary[arrows.meta,bending]
\usetikzlibrary{positioning}
\usetikzlibrary{patterns}
\usepackage{subcaption}

\hypersetup{
	colorlinks=true,
	citecolor=blue,
	linkcolor=blue,
	filecolor=magenta,      
	urlcolor=cyan,
}

\newtheorem{definition}{Definition}
\newtheorem{lemma}{Lemma}
\newtheorem{proposition}{Proposition}
\usepackage{thmtools}
\newtheorem{example}{Example}
\newtheorem{theorem}{Theorem}
\newtheorem{corollary}{Corollary}
\newtheorem{remark}{Remark}

\newcommand{\diff}{\,\mathrm{d}}

\usepackage{setspace}
\usepackage[margin=1.00in]{geometry}

\begin{document}

\title{A new method to construct high-dimensional copulas with Bernoulli and Coxian-2 distributions}
\author{Christopher Blier-Wong\thanks{Corresponding author, \href{mailto:chblw@ulaval.ca}{chblw@ulaval.ca}},
	Hélène Cossette, Sébastien Legros and Etienne Marceau\\ École d'actuariat, Université Laval, Québec, Canada}
\date{September 21 2022}
\maketitle

\begin{abstract}
We propose an approach to construct a new family of generalized Farlie-Gumbel-Morgenstern (GFGM) copulas that naturally scales to high dimensions. A GFGM copula can model moderate positive and negative dependence, cover different types of asymmetries, and admits exact expressions for many quantities of interest such as measures of association or risk measures in actuarial science or quantitative risk management. More importantly, this paper presents a new method to construct high-dimensional copulas based on mixtures of power functions, and may be adapted to more general contexts to construct broader families of copulas. We construct a family of copulas through a stochastic representation based on multivariate Bernoulli distributions and Coxian-2 distributions. This paper will cover the construction of a GFGM copula, study its measures of multivariate association and dependence properties. We explain how to sample random vectors from the new family of copulas in high dimensions. Then, we study the bivariate case in detail and find that our construction leads to an asymmetric modified Huang-Kotz FGM copula. Finally, we study the exchangeable case and provide some insights into the most negative dependence structure within this new class of high-dimensional copulas. 
\end{abstract}

\textbf{Keywords: Generalized Farlie-Gumbel-Morgenstern copulas, stochastic representation, high dimensional copulas, high-dimensional simulation, measures of multivariate association}

\section{Introduction}\label{sec:intro}

Bivariate Farlie-Gumbel-Morgenstern (FGM) distributions have a rich history in dependence modelling and copula theory. FGM copulas are derived from the FGM distributions studied in \cite{farlie1960PerformanceCorrelationCoefficients, gumbel1960BivariateExponentialDistributions, morgenstern1956EinfacheBeispieleZweidimensionaler}. Henri Eyraud also studied distributions of the same shape in \cite{eyraud1936principes}, in a work that predates the seminal work of  \cite{sklar1959FonctionsRepartitionDimensions} by many years. 

The FGM copulas appear in most standard references on dependence and copulas, see, e.g., \cite{kotz2001correlation, nelsen2006IntroductionCopulas, ruschendorf2013MathematicalRiskAnalysis, joe2015DependenceModelingCopulas, durante2015PrinciplesCopulaTheory}. The expression of a bivariate FGM copula is 
\begin{equation}\label{eq:fgm-biv}
	C(u, v) = uv(1 + \theta (1-u)(1-v)),
\end{equation}
for $\theta \in [-1, 1]$ and $(u, v) \in [0, 1]^2$. The value of $\theta$ is a proper dependence parameter; increasing (decreasing) $\theta$ induces positive (negative) associations between the margins. The importance and popularity of bivariate FGM copulas stem partly from their simple shape with quadratic sections, which lead to exact expressions for quantities of interest in dependence modelling and in applications of copulas. For instance, measures of multivariate association are often available in closed-form (see, for instance, \cite{nelsen2006IntroductionCopulas, genest2007EverythingYouAlways}), as well as risk measures for many stochastic models in risk management and actuarial science (examples include \cite{barges2011MomentsAggregateDiscounted,chadjiconstantinidis2014RenewalRiskProcess,woo2013NoteDiscountedCompound}). For this reason, FGM copulas have been important to illustrate dependence concepts in copula theory and applications, contributing to the popularity of copulas in many domains. 

Multivariate FGM copulas are derived from multivariate FGM distributions, introduced in \cite{johnson1975GeneralizedFarliegumbelmorgensternDistributions} and further studied in \cite{cambanis1977PropertiesGeneralizationsMultivariate}; they have an expression given by 
\begin{equation}\label{eq:fgm-natural}
	C(u_1, \dots, u_d) = \prod_{m = 1}^d u_m \left(1 + \sum_{k = 1}^d \sum_{1\leq j_1 < \dots < j_k \leq d} \theta_{j_1 \dots j_k} \overline{u}_{j_1}\dots \overline{u}_{j_k}\right),
\end{equation}
for $(u_1, \dots, u_d) \in [0, 1]^d$, where $\overline{u} = 1 - u$. A $d$-variate FGM copula has $2^d - d - 1$ parameters that we denote by the vector 
\begin{equation}\label{eq:thetas}
    \boldsymbol{\theta} = (\theta_{j_1 \dots j_k} : 1 \leq j_1 < \dots < j_k \leq d, k \in \{1, \dots, d\}).
\end{equation}
The expression in \eqref{eq:fgm-natural} is a valid copula if $\boldsymbol{\theta}$ belongs in the intersection of halfspaces
\begin{equation}\label{eq:Td}
	\left\{\boldsymbol{\theta} \in [-1, 1]^{2^d - d - 1} : 1 + \sum_{k = 2}^d \sum_{1 \leq j_1 < \dots < j_k \leq d} \theta_{j_1 \dots j_k} \varepsilon_{j_1}\dots \varepsilon_{j_k} \geq 0\right\},
\end{equation}
for $\{\varepsilon_{j_1}, \dots, \varepsilon_{j_k}\} \in \{0, 1\}^{2^d - d - 1}$. As opposed to bivariate copulas, the copula parameters in \eqref{eq:fgm-natural} are not easy to interpret (it isn't clear if increasing or decreasing a parameter in $\boldsymbol{\theta}$ leads to more positive or negative dependence) and it is tedious to determine if a given set of parameters satisfies the conditions of \eqref{eq:Td}. Further, it isn't obvious how to simulate observations of random vectors from these high-dimensional copulas, which severely limits the applications of high-dimensional FGM copulas for applications where Monte Carlo simulations are required. Such disadvantages contribute to the sparse literature on high-dimensional FGM copulas in multivariate analysis and in applications of copulas to other sciences. Exceptions include \cite{gijbels2021SpecificationMultivariateAssociation} who develop closed-form expressions for many measures of multivariate association, \cite{barges2009TVaRbasedCapitalAllocation, cossette2013MultivariateDistributionDefined} who study risk aggregation of dependent exponential and mixed Erlang distributions, and \cite{cossette2019CollectiveRiskModels} who study sums of a random number of random variables (rvs) under dependence. 

The authors of \cite{blier-wong2022stochastic} have recently uncovered a relationship between the family of high-dimensional FGM copulas and multivariate Bernoulli distributions with symmetric margins. In particular, they provide a stochastic representation of random vectors whose copula is FGM. It turns out that the dependence structure of FGM copulas relies entirely on an underlying vector of symmetric Bernoulli rvs. This fact allows one to shed light on the dependence structure of the FGM copulas by interpreting the dependence structure of the underlying vector of symmetric Bernoulli rvs, a simpler task since the support of the probability mass function (pmf) of multivariate symmetric Bernoulli distributions is discrete and finite. Relying on the vast literature on various properties of multivariate Bernoulli distributions, the stochastic representation of FGM copulas has led to a much deeper understanding of this family of copulas in high dimensions. In particular, the stochastic representation leads to comparison methods for random vectors under dependence orders, enables constructing subfamilies of high-dimensional FGM copulas with a limited number of parameters, and enables high-dimensional sampling procedures. Some subfamilies scale to countably infinite dimensions and rely on a small number of dependence parameters. Further, \cite{blier-wong2022ExchangeableFGMCopulas} study the subfamily of exchangeable FGM copulas. New results on risk measures in actuarial science also rely on the stochastic representation, see \cite{blier-wong2022RiskAggregationFGMa} for aggregate rvs and \cite{blier-wong2022CollectiveRiskModels} for sums of a random number of rvs. 

The simple shape of FGM copulas has downsides, including a moderate strength of admissible dependence, and for bivariate copulas, being symmetric about its major and minor diagonals. To circumvent these disadvantages, a large literature has developed around extending FGM copulas. 
%A few early examples include \cite{huang1984correlation, huang1999modifications}. 
For a comprehensive review of bivariate FGM copulas and their generalizations, we refer the reader to \cite{bairamov2013huang, saminger-platz2021impact}. Some examples of these generalizations include \cite{huang1984correlation,
	huang1999modifications,
	lai2000NewFamilyPositive,
	bairamov2002dependence,
	rodriguez-lallena2004NewClassBivariatea,
	amblard2009new, kim2011GeneralizedBivariateCopulas, durante2013BivariateCopulasGenerateda,pathak2016NoteGeneralizedFarlieGumbelMorgenstern,
	hurlimann2017comprehensive,
	komornik2017DependenceMeasuresPerturbations,
	cote2019dependence,
	bekrizadeh2022generalized,
	ebaid2022NewExtensionFGM}.
These copulas are usually constructed by validating the existence of the copula. One of these conditions in two dimensions includes the $2$-increasing property, requiring that 
\begin{equation}\label{eq:2-increasing}
	C(u_2, v_2) - C(u_1, v_2) - C(u_2, v_1) + C(u_1, v_1) \geq 0
\end{equation}
for every $0 \leq u_1 < u_2 \leq 1$ and $0\leq v_1 < v_2 \leq 1$. This property is often tedious to prove in two dimensions, and its high-dimensional generalization (the $d$-increasing property that we will state in Section \ref{ss:copulas}), is even harder. For this reason, there are very few generalized FGM copulas in high dimensions, exceptions include \cite{dolati2006some, rodriguez2010multivariate}.

The objective of this paper is to propose a new strategy to construct high-dimensional copulas that lets one control the shape and the strength of dependence. Our starting point is the stochastic representation of FGM copulas from \cite{blier-wong2022stochastic}, but relaxing the assumption that the underlying multivariate Bernoulli rvs be symmetric. The new family of copulas allows one to compare risks under stochastic orders, and admits exact representations for association measures and sampling methods. More importantly, it is simple to interpret the shape of dependence based on an underlying Bernoulli random vector. The method leads to subfamilies that have few parameters, even in countably infinite dimensions. Further, as opposed to many high-dimensional copulas, our family of copulas can admit negative dependence if the underlying multivariate Bernoulli random vector exhibits negative dependence. Finally, since we construct the copula with a stochastic representation, we do not need to verify the $d$-increasing property of copulas: selecting an appropriate distribution for a multivariate Bernoulli random vector will guarantee the existence of the corresponding copula. 

For the bivariate case, it is convenient to work with an algebraic representation of the FGM copula with a genuine dependence parameter. This will lead to simpler methods to understand the geometric properties, the shape of the dependence, etc. Within the algebraic construction, we recover  an asymmetric modified Huang-Kotz FGM copula. For dimensions higher than two, it is not convenient to work with dependence parameters since constraints on the set of parameters may be difficult to satisfy. It is more practical to define and study the properties of the new family of copulas using a stochastic representation rather than an algebraic representation. The stochastic representation guarantees the existence of the copula without the need to validate the $d$-increasing property. 

The remainder of this paper is structured thusly. Section \ref{sec:preliminaries} presents preliminary notions required within this paper. Section \ref{sec:copula} introduces the new family of copulas and studies its properties. In Section \ref{sec:bivariate}, we focus on the bivariate copula and obtain a stochastic representation of one of the copulas studied in \cite{huang1999modifications}. Section \ref{sec:subfamilies} details some properties of the subfamily corresponding to exchangeability. We discuss our findings and propose future works in Section \ref{sec:conclusion}.

\section{Preliminaries}\label{sec:preliminaries}

\subsection{Notation}

We use capital letters for random variables and lowercase letters for their outcomes. The lower case letter $p$ is reserved for probabilities, hence we always have $p \in (0, 1)$. The letter $I$ corresponds to Bernoulli rvs, while the letter $U$ denotes a standard uniformly distributed rv. Bold capital and lower case letters correspond to vectors of rvs and of numbers, implicitly the dimension of the vectors will be $d \in\{2, 3, \dots\}$. It follows that $i \in \{0, 1\}$, $\boldsymbol{i} \in \{0, 1\}^d$, $u \in [0, 1]$ and $\boldsymbol{u} \in [0, 1]^d$. We denote respectively as $F_{X}$ and $F_{\boldsymbol{X}}$ the cumulative distribution function (cdf) and joint cdf of $X$ and $\boldsymbol{X}$. Equivalently, $f_{X}$ and $f_{\boldsymbol{X}}$ correspond to probability density functions or pmfs, depending on the support of $X$ or $\boldsymbol{X}$. All operations such as $\boldsymbol{x} \boldsymbol{y}$, $\boldsymbol{x}^{\boldsymbol{y}}$ or $\boldsymbol{x} \leq \boldsymbol{y}$ are meant component-wise. 

\subsection{Copulas}\label{ss:copulas}

This paper deals with $d$-variate copulas, which correspond to $d$-variate cdfs (restricted to $[0, 1]^d$) with standard uniform margins. The term ``copula'', dating back to the seminal work of \cite{sklar1959FonctionsRepartitionDimensions}, refers to linking $d$ univariate cdfs into a cdf for a $d$-variate vector of uniformly distributed rvs. A $d$-variate copula is a function $C: [0,1]^d \to [0, 1]$ satisfying the boundary conditions $C(u_1, \dots, u_d) = 0$ if any $u_j = 0$, for $j \in {1, \dots, d}$ and $C(u_1, \dots, u_d) = u_j$ if $u_k = 1$ for all $k \in {1, \dots, d}$ and $k \neq j$. Further, copulas must be $d$-increasing on $[0, 1]^d$, meaning that
\begin{equation}\label{eq:d-increasing}
	\sum_{i_1 = 1}^2\dots \sum_{i_d = 1}^2 (-1)^{i_1 + \dots + i_d}C(u_{1i_1}, \dots, u_{di_d}) \geq 0,
\end{equation}
for all $0 \leq u_{j1} < u_{j2} \leq 1$ and $j \in \{1, \dots, d\}.$

One major hurdle in constructing $d$-variate copulas is constructing a function satisfying the $d$-increasing property in \eqref{eq:d-increasing}. Indeed, much of the work in constructing families of extended FGM copulas proposed in the introduction is in determining the set of admissible parameters such that the copulas satisfy the $d$-increasing property. As the dimension $d$ increases, the $d$-increasing property becomes harder to verify, and the parameters have much more constraints to satisfy. This in turn makes it harder to interpret the copula parameters, notably the effect of changing the parameters on the strength of dependence. 

\subsection{Exponential FGM distributions and their stochastic representation}

In this paper, we propose a generalization of the FGM copula. One may extract FGM copulas from multivariate exponential distributions. Let
\begin{equation}\label{eq:cdf-exp-fgm}
	F(\boldsymbol{x}) = \prod_{m = 1}^d (1 - e^{-x_m}) \left(1 + \sum_{k = 2}^d \sum_{1 \leq j_1 < \dots < j_k \leq d}\theta_{j_1 \dots j_k}\prod_{l = 1}^k e^{-x_{j_l}}\right),
\end{equation}
where the parameters $\boldsymbol{\theta}$ must satisfy \eqref{eq:Td} and for $\boldsymbol{x} \in \mathbb{R}^d_+$. Applying Sklar's Theorem to \eqref{eq:cdf-exp-fgm}, one obtains the expression of the $d$-variate FGM copula in \eqref{eq:fgm-natural}. 

In \cite{blier-wong2022stochastic}, the authors find a one-to-one correspondence between the family of exponential FGM distributions and symmetric multivariate Bernoulli distributions. We recall the following theorem from that paper.  
\begin{theorem}\label{thm:X-FGM}
	Let $\boldsymbol{I}$ be a symmetric Bernoulli random vector, $\boldsymbol{Y}_0$ be a vector of independent exponentially distributed rvs with mean 1/2 and $\boldsymbol{Y}_1$, a vector of independent exponentially distributed rvs with mean 1. Further, assume that $\boldsymbol{I}$, $\boldsymbol{Y}_0$ and $\boldsymbol{Y}_1$ are independent of each other. Let 
	\begin{equation}\label{eq:X-FGM}
		\boldsymbol{X} = \boldsymbol{Y}_0 + \boldsymbol{I} \boldsymbol{Y}_1
	\end{equation}
	and
	$$\theta_{j_1 \dots j_k} = \sum_{i_{j_1}, \dots, i_{j_k} \in \{0, 1\}^k} (-1)^{i_{j_1} + \dots + i_{j_k}}f_{I_{j_1},\dots, I_{j_k}}(i_{j_1}, \dots, i_{j_k}),$$
	for $j \leq j_1 < \dots < j_k \leq d$ and $k \in \{2, \dots, d\}$. Then, the cdf of $\boldsymbol{X}$ can be written as \eqref{eq:cdf-exp-fgm}. 
\end{theorem}

In this paper, we rely on the construction of Theorem \ref{thm:X-FGM} and generalize the assumptions of the representation to propose a new family of copulas. Before proceeding to the novel family of copulas, we will require some alternate representations of univariate exponential and uniform rvs, that we will introduce next. 

% \begin{itemize}
% 	\item Définition d'une copule
% 	\item $\overline{u} = 1 - u$
% \end{itemize}

% We define the following copula transformations:
% \begin{align*}
% 	C^{\text{uflip}}(u, v) &= v - C(\overline{u}, v)\\
% 	C^{\text{vflip}}(u, v) &= u - C(u, \overline{v})\\
% 	C^{\text{surv}}(u, v) &= u + v - 1 + C(\overline{u}, \overline{v})
% \end{align*}

% We say that $C$ is symmetric if $C(u, v) = C(v, u)$

% We say that $C$ is radially symmetric if $C(u, v) = C^{\text{surv}}(u, v)$

\subsection{Some univariate results}

Following  \cite[Section 2.3]{klugman2013loss}, \cite[Chapter 9]{asmussen2010ruin} or \cite[Example 3.4]{bladt2005review}
we define Coxian-2 distributions through their Laplace-Stieltjes transforms (LST), 
\begin{equation}\label{eq:lst-cox2}
	\mathcal{L}_{X}(t) = (1-p)\frac{\beta_1}{t + \beta_1} + p\frac{\beta_1}{t + \beta_1}\frac{\beta_2}{t + \beta_2},
\end{equation}
for $\beta_1 > 0, \beta_2 > 0$ and $\beta_1 \neq \beta_2$. 

We aim to find the conditions under which a Coxian-2 distribution leads to a standard exponential distribution; the following proposition presents one such construction.
\begin{proposition}\label{prop:marginal-exp-nonsymmetric}
	Let $X$ be a Coxian-2 distribution with $p \in (0, 1)$, $\beta_1 = 1/(1-p)$ and $\beta_2 = 1$. Then, $X \sim Exp(1)$. 
\end{proposition}
\begin{proof}
	We prove the statement through LSTs. Replacing the parameters of Proposition \ref{prop:marginal-exp-nonsymmetric} within the LSTs of Coxian-2 distributions in \eqref{eq:lst-cox2}, we have 
	\begin{align*}
		\mathcal{L}_X(t) &= \frac{(1-p)^{-1}}{(1-p)^{-1} + t}\left(1 - p + p \frac{1}{1 + t}\right)= \frac{1}{1 + t(1-p)} \left(\frac{(1-p)(1+t)}{1+t} + \frac{p}{1 + t}\right)\\
		&= \frac{1}{1 + t(1-p)}\left(\frac{1 + t(1-p)}{1 + t}\right) = \frac{1}{1 + t}, \quad t \geq 0,
	\end{align*}
	which is the LST of an exponential distribution with parameter 1. 
\end{proof}

From the results of Proposition \ref{prop:marginal-exp-nonsymmetric}, we have the following stochastic representation. Let $Y_0 \sim Exp\left((1-p)^{-1}\right)$, $Y_1 \sim Exp(1)$ and $I \sim Bern(p)$, and let $Y_0, Y_1$ and $I$ be independent from each other. Then, 
\begin{equation}\label{eq:representation-x-univariate}
	Y_0 + I Y_1 \overset{\mathcal{D}}{=} X,
\end{equation}
where $X \sim Exp(1)$.

Proposition \ref{prop:marginal-exp-nonsymmetric} lets one construct an alternative representation of uniform rvs. Let $Y_2$ follow a standard exponential distribution, and $U$, $U_0$ and $U_1$ be independent standard uniform rvs and $I$ be Bernoulli distributed with success probability $p$. Applying the probability integral transform (see, e.g., \cite[Theorem 2.1.10]{casella2002statistical}), we have that $1-\exp\{-X\} \overset{\mathcal{D}}{=} U$, and that $1 - \exp\{-Y_0 - IY_1\} \overset{\mathcal{D}}{=} 1 - \exp\{-(1-p)Y_2 - IY_1\} \overset{\mathcal{D}}{=} 1 - \exp\{-Y_2\}^{1-p}\exp\{-Y_1\}^I\overset{\mathcal{D}}{=} U_0^{1-p}U_1^{I}$. It follows from \eqref{eq:representation-x-univariate} that
\begin{equation}\label{eq:u-stochastic-cox2}
	U_0^{1-p} U_1^{I}\overset{\mathcal{D}}{=} U.
\end{equation}
One may verify that the cdf of $U$ is $u$, for $u\in[0, 1]$. However, to explore copulas constructed within this paper, it will be useful to derive the cdf associated with the stochastic representation in  \eqref{eq:u-stochastic-cox2}. 
\begin{lemma}\label{lemma:u-stochastic-marginal}
	The cdf of $U$ from the representation in \eqref{eq:u-stochastic-cox2} is
	\begin{equation}\label{eq:cdf-u-marginal}
		F_{U}(u) = E\left[(1-I) u^{(1-p)^{-1}} + I \left(\frac{u}{p} - \frac{1-p}{p} u^{(1-p)^{-1}}\right) \right], \quad u \in[0, 1].
	\end{equation}	
\end{lemma}
\begin{proof}
	We have
	\begin{align*}
		F_U(u) &= \Pr(U \leq u) = E_{I}\left[\Pr(U \leq u \vert I)\right] = E_{I}\left[\Pr\left(U_0 ^{1-p} U_1 ^{I} \leq u \vert I\right)\right] \\
		&= \Pr(I = 0) \Pr\left(U_0 ^{1-p} \leq u\right) + \Pr(I = 1) \Pr\left(U_0 ^{1-p}U_1 \leq u\right),
		% &= (1-p) \Pr\left(U_0 \leq u^{(1-p)^{-1}}\right) + p E_{U_1}\left[\Pr\left(U_0 ^{1-p}U_1 \leq u\right)\right]\\
%		&= \Pr(I = 0) u^{(1-p)^{-1}} + \Pr(I = 1) E_{U_1}\left[\Pr\left(\left.U_0 \leq \left(\frac{u}{U_1}\right)^{(1-p)^{-1}} \right\vert U_1\right)\right].	
	\end{align*}
	which, by conditioning on $U_1$, becomes 
	$$F_U(u) = \Pr(I = 0) u^{(1-p)^{-1}} + \Pr(I = 1) E_{U_1}\left[\Pr\left(\left.U_0 \leq \left(\frac{u}{U_1}\right)^{(1-p)^{-1}} \right\vert U_1\right)\right].$$
	To solve the expected value, we must handle the cases where $u/U_1 >1$ and where $u/U_1 \leq1$ differently as follows:
	\begin{align*}
		\Pr(U \leq u) & = \Pr(I = 0)u^{(1-p)^{-1}} + \Pr(I = 1)\left(\int_0^u \diff v + \int_u^1 \left(\frac{u}{v}\right)^{(1-p)^{-1}} \diff v\right) \\
		              & = \Pr(I = 0)u^{(1-p)^{-1}} + \Pr(I = 1)\left(\frac{u}{p} - \frac{1-p}{p} u^{(1-p)^{-1}}\right).
	\end{align*}
	Writing the last equality as the expected value over $I$ completes the proof. 
\end{proof}

\section{A new high-dimensional copula}\label{sec:copula}

This section introduces the copula studied in this paper. We start with the stochastic representation of FGM copulas from Theorem \ref{thm:X-FGM}, then extend this family using the representation of exponential rvs in Proposition \ref{prop:marginal-exp-nonsymmetric}. 

\subsection{Representation}\label{ss:representation}

Let $\boldsymbol{I} = (I_1, \dots, I_d)$ be a $d$-variate Bernoulli random vector, where $\Pr(I_j = 1) = p_j$ for $j \in \{1, \dots, d\}$. For notational convenience, we will denote the vector of probabilities $\boldsymbol{p} = (p_1, \dots, p_d)$. Let $\boldsymbol{Y}_0$ be a vector of exponentially distributed rvs, where the $j$th margin has mean $1-p_j$, for $j \in \{1, \dots, d\}$. Let $\boldsymbol{Y}_1$ be a vector of independent exponentially distributed rvs with mean 1. Further assume that the vectors $\boldsymbol{I}$, $\boldsymbol{Y}_0$ and $\boldsymbol{Y}_1$ are independent of each other. We define the random vector 
\begin{equation}\label{eq:stochastic-form}
	\boldsymbol{X} = \boldsymbol{Y}_0 + \boldsymbol{I} \boldsymbol{Y}_1.
\end{equation}
Following Proposition \ref{prop:marginal-exp-nonsymmetric}, $\boldsymbol{X}$ forms a vector of rvs with standard exponential margins. Applying the probability integral transform to each component in \eqref{eq:stochastic-form}, we obtain an equivalent representation for uniform margins. Let $\boldsymbol{U}_0 = (U_{0, 1}, \dots, U_{0, d})$ and $\boldsymbol{U}_1 = (U_{1, 1}, \dots, U_{1, d})$ be random vectors of independent uniform rvs and $\boldsymbol{I}$ be a $d$-variate Bernoulli random vector with vector of probabilities $\boldsymbol{p}$. The random vectors $\boldsymbol{I}$, $\boldsymbol{U}_0$ and $\boldsymbol{U}_1$ are independent of each other. Further, define
\begin{equation}\label{eq:representation-u}
	\boldsymbol{U} \overset{\mathcal{D}}{=} \boldsymbol{U}_0 ^{\boldsymbol{1} - \boldsymbol{p}} \boldsymbol{U}_1^{\boldsymbol{I}},
\end{equation}
where $\boldsymbol{1}$ is a vector of ones. Our objective in the remainder of this paper is to study the family of copulas that correspond to cdfs for random vectors in \eqref{eq:representation-u}. 
\begin{theorem}\label{thm:copula}
	The copula associated with the cdf of the random vector in \eqref{eq:representation-u} is 
	\begin{equation}\label{eq:copula-stochastic}
		C(\boldsymbol{u}) = E\left[\prod_{m = 1}^{d}\left(u_m ^{(1-p_m)^{-1}} - I_m \left\{\frac{u_m^{(1-p_m)^{-1}} - u_m}{p_m}\right\}\right)\right], \quad \boldsymbol{u}\in [0, 1]^d.
	\end{equation} 
\end{theorem}
\begin{proof}
We have $$\Pr(\boldsymbol{U} \leq \boldsymbol{u}) = \Pr\left(\boldsymbol{U}_0 ^{\boldsymbol{1} - \boldsymbol{p}} \boldsymbol{U}_1^{\boldsymbol{I}} \leq \boldsymbol{u}\right).$$
Conditioning on $\boldsymbol{I}$ and by independence of $\boldsymbol{I}, \boldsymbol{U}_0$ and $\boldsymbol{U}_1$, we have
$$\Pr(\boldsymbol{U} \leq \boldsymbol{u}) = E_{\boldsymbol{I}}\left[\prod_{m = 1}^d \Pr\left(U_{0, m} ^{1 - p_m} U_{1, m}^{I_m} \leq u_m \vert \boldsymbol{I}\right)\right].$$
Applying Lemma \ref{lemma:u-stochastic-marginal} on each multiple, we obtain 
$$\Pr(\boldsymbol{U} \leq \boldsymbol{u}) = E\left[\prod_{m = 1}^{d} \left((1-I_m)u_m^{(1-p_m)^{-1}} + I_m\left\{\frac{u_m}{p_m} - \frac{1-p_m}{p_m} u_m^{(1-p_m)^{-1}} \right\}\right)\right].$$
We obtain \eqref{eq:copula-stochastic} by rearranging the terms of the last equality. 
\end{proof}
In the remainder of this paper, we denote by $\mathcal{C}^{GFGM}$ the class of copulas with expressions as in Theorem \ref{thm:copula}. One may also idenitfy a natural representation of $C\in \mathcal{C}^{GFGM}$. We obtain the following corollary by factoring and expanding the product in Theorem \ref{thm:copula}.
\begin{corollary}
The expression of the copula in \eqref{eq:copula-stochastic} is also
%\begin{equation}\label{eq:copula-natural}
%	C(\boldsymbol{u}) = \prod_{m = 1}^d u_m^{\frac{1}{1-p_m}} \left(1 + \sum_{k = 1}^d (-1)^k\sum_{1 \leq j_1 < \dots < j_k \leq d} \frac{E[I_{j_1}\dots I_{j_k}]}{p_{j_1}\dots p_{j_k}} \left(1 - u_{j_1}^{-\frac{p_{j_1}}{1-p_{j_1}}}\right) \dots \left(1 - u_{j_k}^{-\frac{p_{j_k}}{1-p_{j_k}}}\right)\right),
%\end{equation}
\begin{equation}\label{eq:copula-natural}
	C(\boldsymbol{u}) = \prod_{m = 1}^d u_m \left(1 + \sum_{k = 1}^d \sum_{1 \leq j_1 < \dots < j_k \leq d} \nu_{j_1\dots j_k} \left(1 - u_{j_1}^{\frac{p_{j_1}}{1-p_{j_1}}}\right) \dots \left(1 - u_{j_k}^{\frac{p_{j_k}}{1-p_{j_k}}}\right)\right),
\end{equation}
for $\boldsymbol{u} \in [0, 1]^d,$ where
$$\nu_{j_1\dots j_k} = E\left[\prod_{n = 1}^{k} \frac{I_{j_n} - p_{j_n}}{p_{j_n}}\right]$$
for $1\leq j_1 < \dots < j_k \leq d$ and $k \in \{2, \dots, d\}$.
\end{corollary}
\begin{remark}
	We present remarks related to the class $\mathcal{C}^{GFGM}$ of copulas.
	\begin{enumerate}
	\item It follows that \eqref{eq:copula-stochastic} satisfies all of the conditions for the existence of copulas as presented in \eqref{ss:copulas}. Therefore, one does not need to go through the tedious steps of verifying the $d$-increasing property. 
	\item Copulas in $\mathcal{C}^{GFGM}$ have linear sections raised to a power, hence will lead to closed-form expressions for many quantities of interest in dependence modelling, and in applications like actuarial science.
	\item We say that $\boldsymbol{p}$ is a shape parameter of the copula since its value determines whether the copula will have more mass on the upper or lower tails. This statement will become clear in Section \ref{sec:bivariate}, where we present the density function for the bivariate copulas with different values of $p_1$ and $p_2$. Also, note that each vector of $\boldsymbol{p}$ leads to a subfamily of copulas, and it does not make sense to compare the strength of dependence between two copulas with different vectors $\boldsymbol{p}$.  
    \item The copula in \eqref{eq:copula-natural} has dependence parameters, akin to the vector of parameters $\boldsymbol{\theta}$ in the family of FGM copulas. However, we advise against using the natural representation of the copula in \eqref{eq:copula-natural} because it is not easy to interpret the strength of dependence, nor is it easy to verify the conditions under which the set of parameters $\nu_{j_1\dots j_k}$ for $1\leq j_1<\dots< j_k \leq d$ and $k \in \{2, \dots, d\}$ yield a copula that satisfies the $d$-increasing property. Further, the pmf of multivariate Bernoulli distributions can be expressed as points in a convex hull \cite{fontana2018RepresentationMultivariateBernoulli} and the extreme points of the convex hull can sometimes be found analytically \cite{fontana2021ModelRiskCredit, fontana2022HighDimensionalBernoulli}. It follows that the copula in \eqref{eq:copula-stochastic} will share these properties, which will let us characterize the possible dependence structures for a fixed vector of probabilities $\boldsymbol{p}$.
		\item One may alternate between the stochastic representation in \eqref{eq:copula-stochastic} and the natural representation in \eqref{eq:copula-natural} by alternating between pmfs and ordinary moments of multivariate Bernoulli distributions as described in \cite{teugels1990representations}. 
	\end{enumerate}
\end{remark}

\subsection{Simulation}

Exact expressions for risk measures or optimization problems derived from multivariate stochastic models appear very rarely in practice. A strategy, prevalent in financial applications and quantitative risk management, is to resort to Monte Carlo simulation (see the preface of \cite{mai2012simulating} for a discussion). Such applications require efficient simulation algorithms in high dimensions. 

The main advantage of the family of copulas that we propose in this paper is that many risk measures are available in closed form. However, for practical applications, it is relevant to have a simulation procedure. Until recently, simulation from FGM copulas has relied on the conditional method (also called the Rosenblatt transform), see, e.g., \cite[Section 6.9.1]{joe2015DependenceModelingCopulas} for details on the method. The conditional inverse of $d$-variate FGM copulas are available in closed-form, hence one may simulate random vectors with FGM copulas using the conditional method, see \cite{ota2021effective} and \cite{blier-wong2022stochastic} for details. However, the conditional method does not scale well to high dimensions since one must successively sample from consecutive margins. High-dimensional simulation becomes feasible once one has a stochastic representation for the random vector $\boldsymbol{U}$. We present, in Algorithm \ref{algo:sample-stochastic}, a simulation procedure based on the representation in \eqref{eq:representation-u}.
\begin{algorithm}
	\KwIn{Number of simulations $n$, pmf $f_{\boldsymbol{I}}$}
	\KwOut{Set of simulations}
	\nl \For{$l = 1, \dots, n$}{
		\nl Generate $n$ independent random vectors $\boldsymbol{I}^{(l)}$, $\boldsymbol{U}_0^{(l)}$ and $\boldsymbol{U}_1^{(l)}$\;
		\nl \For{$m = 1, \dots, d$}{
			\nl Compute $U_{m}^{(l)}=\left(U_{0,m}^{(l)}\right)^{1-p_m}\left(U_{1,m}^{(l)}\right)^{I_m^{(l)}}$\; 
	}}
	\nl Return $\boldsymbol{U}^{(l)} = \left(U_1^{(l)}, \dots, U_d^{(l)}\right),  l = 1, \dots, n$.
	\caption{Stochastic simulation method} \label{algo:sample-stochastic}
\end{algorithm}

\subsection{Dependence ordering}\label{ss:dep-order}

To study the effect of dependence on a model, it is useful to use dependence orders. Such orders compare the strength of dependence between random vectors $\boldsymbol{U}$ and $\boldsymbol{U}'$. We consider four dependence orders that are relevant for copulas: the supermodular order (denoted $\preceq_{sm}$), the lower concordance order (denoted $\preceq_{cL}$), the upper concordance order (denoted $\preceq_{cU}$) and the concordance order (denoted $\preceq_{c}$). Details on dependence orders can be found in, e.g., \cite[Chapter 9]{shaked2007StochasticOrders}, \cite[Chapter 3]{muller2002ComparisonMethodsStochastic} or \cite[Chapter 6]{denuit2006ActuarialTheoryDependent}. The supermodular order is defined as follows.
\begin{definition}[Supermodular order] \label{defSupermodularOrder}
	We say that $\boldsymbol{U}$ is smaller than $\boldsymbol{U}'$ under the supermodular order if $E\left\{\phi (U_1, \dots, U_d)\right\} \leq E\left\{ \phi (U_1', \dots, U_d')\right\} $ for all supermodular functions $\phi $, given that the expectations exist. A function $\phi :\mathbb{R}^{d}\rightarrow \mathbb{R}$ is said to be supermodular if
	\begin{eqnarray*}
		&&\phi (x_{1},\ldots,x_{i}+\varepsilon ,\ldots,x_{j}+\delta ,\ldots,x_{d})-\phi
		(x_{1},\ldots,x_{i}+\varepsilon ,\ldots,x_{j},\ldots,x_{d}) \\
		&\geq &\phi (x_{1},\ldots,x_{i},\ldots,x_{j}+\delta ,\ldots,x_{d})-\phi
		(x_{1},\ldots,x_{i},\ldots,x_{j},\ldots,x_{d})
	\end{eqnarray*}
	holds for all $(x_1, \dots, x_d)\in \mathbb{R}^{d}$, $1\leq
	i< j\leq d$\ and all $\varepsilon$, $\delta >0$.
\end{definition}
Further, we say that $\boldsymbol{U}\preceq_{cL}\boldsymbol{U}'$ if $F_{\boldsymbol{U}}(\boldsymbol{u}) \leq F_{\boldsymbol{U}'}(\boldsymbol{u})$ for all $\boldsymbol{u}\in[0, 1]^d$, that $\boldsymbol{U}\preceq_{cU}\boldsymbol{U}'$ if $\overline{F}_{\boldsymbol{U}}(\boldsymbol{u}) \leq \overline{F}_{\boldsymbol{U}'}(\boldsymbol{u})$ for all $\boldsymbol{u}\in[0, 1]^d$. Further, we have $\boldsymbol{U}\preceq_{c}\boldsymbol{U}'$ if both $\boldsymbol{U}\preceq_{cL}\boldsymbol{U}'$ and $\boldsymbol{U}\preceq_{cU}\boldsymbol{U}'$. Also, $\boldsymbol{U}\preceq_{sm}\boldsymbol{U}'$ implies $\boldsymbol{U}\preceq_{c}\boldsymbol{U}'$. 

\begin{theorem}\label{thm:stochastic-orders}
Let $\boldsymbol{U}$ and $\boldsymbol{U}'$ be random vectors constructed using the representation in \eqref{eq:representation-u} and with the random vectors $\boldsymbol{I}$ and $\boldsymbol{I}'$ for some fixed $\boldsymbol{p}$. The following relationships hold:
	\begin{enumerate}
		\item If $\boldsymbol{I} \preceq_{sm} \boldsymbol{I}'$, then 
		$\boldsymbol{U} \preceq_{sm} \boldsymbol{U}'$.
		\item If $\boldsymbol{I} \preceq_{cU} \boldsymbol{I}'$, then $\boldsymbol{U} \preceq_{cU} \boldsymbol{U}'$.
		\item If $\boldsymbol{I} \preceq_{cL} \boldsymbol{I}'$, then $\boldsymbol{U} \preceq_{cL} \boldsymbol{U}'$.
		\item If $\boldsymbol{I} \preceq_{c} \boldsymbol{I}'$, then $\boldsymbol{U} \preceq_{c} \boldsymbol{U}'$.
	\end{enumerate}
\end{theorem}
\begin{proof}
	The proof is identical to the proof of Theorem 4.2 of \cite{blier-wong2022stochastic}.
\end{proof}

Theorem \ref{thm:stochastic-orders} has many implications for applications of copulas, see, for instance, \cite{li2013stochastic}, \cite[Chapter 6]{denuit2006ActuarialTheoryDependent} or \cite[Chapter 8]{muller2002ComparisonMethodsStochastic}. 

An important result in supermodular ordering (see, e.g., \cite{dhaene2002ConceptComonotonicityActuarial}, \cite[Section 6.3.7]{denuit2006ActuarialTheoryDependent}, \cite[page 119]{ruschendorf2013MathematicalRiskAnalysis} for discussions) is that the comonotonic random vector, denoted $\boldsymbol{I}^c$, corresponds to the most positively dependent random vector under the supermodular order, that is, for some $V \sim Unif(0, 1)$, we have
\begin{equation}\label{eq:i-sm}
	\boldsymbol{I} \preceq_{sm} \boldsymbol{I}^c = (F_{I_1}^{-1}(V), \dots, F^{-1}_{I_d}(V))
\end{equation}
for all $\boldsymbol{I}$. The following theorem presents the stochastic representation for the largest random vector, for a given set of shape parameters, within the class of copulas in $\mathcal{C}^{GFGM}$. 
\begin{theorem}
	Fix some vector of probabilities $\boldsymbol{p}$ and let $\boldsymbol{I}^{c}$ be a vector of comonotonic Bernoulli rvs. Define the random vector 
		\begin{equation}\label{key}
				\boldsymbol{U}^{EPD-GFGM} = \boldsymbol{U}_0^{\boldsymbol{1}-\boldsymbol{p}} \boldsymbol{U}_1 ^{\boldsymbol{I^c}}.
		\end{equation}
	Let $F_{\boldsymbol{U}^{EPD-GFGM}} = C^{EPD-GFGM}$. Then, for all $C \in \mathcal{C}^{GFGM}$ with a vector of probabilities $\boldsymbol{p}$ and $\boldsymbol{U}$, the random vector associated with the copula $C$, we have
	$$\boldsymbol{U} \preceq_{sm}\boldsymbol{U}^{EPD-GFGM}.$$
%	For a given vector of probabilities $\boldsymbol{p}$, the random vector $\boldsymbol{U}$ with representation in \eqref{eq:representation-u} that leads to the extremal positive dependence, denoted $\boldsymbol{U}^{EPD}$, is given by
%	\begin{equation}\label{key}
%		\boldsymbol{U}^{EPD} = \boldsymbol{U}_0^{1-\boldsymbol{p}} \boldsymbol{U}_1 ^{\boldsymbol{I^c}}.
%	\end{equation}
\end{theorem}
\begin{proof}
	The result follows from Theorem \ref{thm:stochastic-orders} and the relation in \eqref{eq:i-sm}.
\end{proof}

In particular, if $p_1 = \dots = p_d = p$, then the pmf of $\boldsymbol{I}^c$ is $\Pr(\boldsymbol{I}^c = \boldsymbol{0}) = 1-p$, $\Pr(\boldsymbol{I}^c = \boldsymbol{1}) = p$ and zero elsewhere. It follows that the copula associated with $\boldsymbol{U}^{EPD-GFGM}$ simplifies to 
\begin{equation}\label{eq:copula-epd}
	C^{EPD}(\boldsymbol{u}) = (1-p)\prod_{m = 1}^{d} u_m ^{1/(1-p)} + p\prod_{m = 1}^{d} \left(\frac{p-1}{p} u_m^{1/(1-p)} -\frac{u_m}{p}\right), \quad \boldsymbol{u} \in [0,1]^d.
\end{equation}

\subsection{Association measures}\label{ss:association}

When dealing with $d$-variate random vectors, it is useful to quantify the degree to which the rvs are associated. Four important measures of multivariate association are $\rho^{cL}$, Spearman's rho derived from average upper orthant dependence $\rho^{cU}$ and Kendall's tau derived from multiplicative total positivity of order 2. Their expressions are given in the following definition. 
\begin{definition}
	Let $C^{\perp}$ be the independence copula, that is, $C^{\perp}(u_1, \dots, u_d) = u_1 \dots u_d$, for all $(u_1, \dots, u_d)\in [0, 1]^d$. Following \cite{nelsen1996NonparametricMeasuresMultivariate} and \cite{nelsen2002ConcordanceCopulasSurvey}, we define four measures of multivariate association:
	\begin{enumerate}
		\item Spearman's rho derived from average lower orthant dependence
		\begin{equation}
			\rho^{cL}(\boldsymbol{U}) = \frac{d + 1}{2^{d}-d-1} \left[2^d\left(\int_{[0, 1]^d} C(\boldsymbol{u}) \diff C^{\perp}(\boldsymbol{u})\right) - 1\right].\label{eq:def-rhocL}
		\end{equation}
		\item Spearman's rho derived from average upper orthant dependence
		\begin{equation}
			\rho^{cU}(\boldsymbol{U}) = \frac{d + 1}{2^{d}-d-1} \left[2^d\left(\int_{[0, 1]^d} C^{\perp}(\boldsymbol{u}) \diff C(\boldsymbol{u})\right) - 1\right].\label{eq:def-rhocU}
		\end{equation}
		\item Spearman's rho derived from concordant dependence
		\begin{equation}
			\rho^{c}(\boldsymbol{U}) = \frac{1}{2} \left(\rho^{cL} + \rho^{cU}\right)
		\end{equation}
		\item Kendall's tau derived from multiplicative total positivity of order 2
		\begin{equation}
			\tau(\boldsymbol{U}) = \frac{1}{2^{d - 1} - 1} \left\{2^d \int_{[0, 1]^d}C(u) \diff C(u) - 1\right\}.\label{eq:def-tau}
		\end{equation}
	\end{enumerate}
\end{definition} 

The family of copulas studied in this paper admits exact expressions for the four measures of multivariate association.
\begin{proposition}\label{prop:measures-association}
	Let $\boldsymbol{U}$ be a $d$-variate random vector with copula as in \eqref{eq:copula-stochastic}. Then, the measures of multivariate association $\rho^{cL}$, $\rho^{cU}$, $\rho^{c}$ and $\tau$ of $\boldsymbol{U}$ are given by
	\begin{align}
		\rho^{cL}(\boldsymbol{U})&= \frac{d + 1}{2^d - d - 1} \left(\left\{\prod_{m = 1}^{d} \frac{1}{2-p_m}\right\}E\left[ \prod_{m = 1}^d \left(2(1-p_m) + I_m\right)\right] - 1\right);\label{eq:rhocL}\\
		\rho^{cU}(\boldsymbol{U}) &= \frac{d + 1}{2^d - d - 1} \left(\left\{\prod_{m = 1}^{d} \frac{1}{2-p_m}\right\} E\left[\prod_{m = 1}^{d}(2-I_m)\right] - 1\right);\label{eq:rhocU}\\
		\rho^{c}(\boldsymbol{U}) &= \frac{d + 1}{2^d - d - 1} \left(\left\{\prod_{m = 1}^{d} \frac{1}{2-p_m}\right\} \frac{1}{2}\left(E\left[ \prod_{m = 1}^d \left(2(1-p_m) + I_m\right) + \prod_{m = 1}^{d}(2-I_m)\right]\right) - 1\right);\label{eq:rhoc}\\
		\tau(\boldsymbol{U}) &= \frac{1}{2^{d-1}-1} \left(2^d \sum_{\boldsymbol{i}\in\{0, 1\}^d} \sum_{\boldsymbol{j}\in\{0, 1\}^d} f_{\boldsymbol{I}}(\boldsymbol{i})f_{\boldsymbol{I}}(\boldsymbol{j})\prod_{m = 1}^{d} \left(\frac{1}{2} - \frac{i_m + j_m}{2p_m} + \frac{j_m(1-p_m) + i_m}{p_m(2-p_m)}\right) - 1\right).\label{eq:tau}
	\end{align}
\end{proposition}
\begin{proof}
	The integral in \eqref{eq:def-rhocL} is 
	\begin{align}
		\int_{[0, 1]^d} C(\boldsymbol{u}) \diff C^{\perp}(\boldsymbol{u}) &= \int_{[0, 1]^d} E\left[\prod_{m = 1}^{d}\left(u_m ^{(1-p_m)^{-1}} - I_m \left\{\frac{u_m^{(1-p_m)^{-1}} - u_m}{p_m}\right\}\right)\right] \diff C^{\perp}(\boldsymbol{u})\nonumber\\
		&= E\left[\prod_{m = 1}^{d} \int_{0}^{1}\left(u_m ^{(1-p_m)^{-1}} - I_m \left\{\frac{u_m^{(1-p_m)^{-1}} - u_m}{p_m}\right\}\right) \diff u_m\right]\nonumber\\
		&= E\left[\prod_{m = 1}^{d} 
		\left(\frac{1}{(1-p_m)^{-1}+1} - I_m \left\{\frac{\frac{1}{(1-p_m)^{-1}+1} - \frac{1}{2}}{p_m}\right\}\right)
		\right].\label{eq:int-rhom}
	\end{align}
	We obtain \eqref{eq:rhocL} by replacing the integral in \eqref{eq:def-rhocL} by \eqref{eq:int-rhom} and simple calculations. To solve the integral in \eqref{eq:rhocU}, we use the chain rule and obtain
	$$\int_{[0, 1]^d} C^{\perp}(\boldsymbol{u}) \diff C(\boldsymbol{u}) = \int_{[0, 1]^d} C^{\perp}(\boldsymbol{u}) E\left[\prod_{m = 1}^{d}\left(\frac{u_m ^{(1-p_m)^{-1}-1}}{1-p_m} - I_m \left\{\frac{\frac{u_m^{(1-p_m)^{-1}-1}}{1-p_m} - 1}{p_m}\right\}\right)\right]\diff C^{\perp}(\boldsymbol{u}).$$
	Inserting the integral within the expectation, we have	
	$$\int_{[0, 1]^d} C^{\perp}(\boldsymbol{u}) \diff C(\boldsymbol{u}) =  E\left[\prod_{m = 1}^{d} \int_{0}^{1}\left(\frac{u_m ^{(1-p_m)^{-1}}}{1-p_m} - I_m \left\{\frac{\frac{u_m^{(1-p_m)^{-1}}}{1-p_m} - u_m}{p_m}\right\}\right)\diff u_m\right]$$
	The remaining integral involes integrating a power and is simple to solve. The result in \eqref{eq:rhoc} is the average of \eqref{eq:rhocL} and \eqref{eq:rhocU}. Finally, we apply once again the clain rule to the the integral in \eqref{eq:def-tau} and obtain
	\begin{align*}
		\int_{[0, 1]^d}C(u) \diff C(u) &= \int_{[0, 1]^d} E\left[\prod_{m = 1}^{d}\left(u_m ^{(1-p_m)^{-1}} - I_m \left\{\frac{u_m^{(1-p_m)^{-1}} - u_m}{p_m}\right\}\right)\right] \times \\
		&  \qquad\qquad E\left[\prod_{m = 1}^{d}\left(\frac{u_m ^{(1-p_m)^{-1}-1}}{1-p_m} - I_m \left\{\frac{\frac{u_m^{(1-p_m)^{-1}-1}}{1-p_m} - 1}{p_m}\right\}\right)\right]\diff C^{\perp}(\boldsymbol{u}) \\
		&= \int_{[0, 1]^d} \left(\sum_{\boldsymbol{i}\in\{0, 1\}^d}
		f_{\boldsymbol{I}}(\boldsymbol{i})\prod_{m = 1}^{d}\left(u_m ^{(1-p_m)^{-1}} - i_m \left\{\frac{u_m^{(1-p_m)^{-1}} - u_m}{p_m}\right\}\right)\right)\times\\
		& \qquad\qquad \left(\sum_{\boldsymbol{i}\in\{0, 1\}^d} f_{\boldsymbol{I}}(\boldsymbol{j})
		\prod_{m = 1}^{d}\left(\frac{u_m ^{(1-p_m)^{-1}-1}}{1-p_m} - j_m \left\{\frac{\frac{u_m^{(1-p_m)^{-1}-1}}{1-p_m} - 1}{p_m}\right\}\right)\right)\diff C^{\perp}(\boldsymbol{u}).	
	\end{align*}
	Expanding the series, we have
	\begin{align*}
		\int_{[0, 1]^d}C(u) \diff C(u) &= \int_{[0, 1]^d} \sum_{\boldsymbol{i}\in\{0, 1\}^d}\sum_{\boldsymbol{j}\in\{0, 1\}^d}
		f_{\boldsymbol{I}}(\boldsymbol{i})f_{\boldsymbol{I}}(\boldsymbol{j})
		\left[\prod_{m = 1}^{d}\left(u_m ^{(1-p_m)^{-1}} - i_m \left\{\frac{u_m^{(1-p_m)^{-1}} - u_m}{p_m}\right\}\right) \times \vphantom{\left\{\frac{\frac{u_m^{(1-p_m)^{-1}-1}}{1-p_m} - 1}{p_m}\right\}}\right. \\
		& \qquad \qquad 
		\left.\prod_{m = 1}^{d}\left(\frac{u_m ^{(1-p_m)^{-1}-1}}{1-p_m} - j_m \left\{\frac{\frac{u_m^{(1-p_m)^{-1}-1}}{1-p_m} - 1}{p_m}\right\}\right)\right]\diff C^{\perp}(\boldsymbol{u})\\
		&= \sum_{\boldsymbol{i}\in\{0, 1\}^d}\sum_{\boldsymbol{j}\in\{0, 1\}^d}f_{\boldsymbol{I}}(\boldsymbol{i})f_{\boldsymbol{I}}(\boldsymbol{j})
		\prod_{m = 1}^{d} \int_{0}^{1}\left[\left(u_m ^{(1-p_m)^{-1}} - i_m \left\{\frac{u_m^{(1-p_m)^{-1}} - u_m}{p_m}\right\}\right) \vphantom{\left\{\frac{\frac{u_m^{(1-p_m)^{-1}-1}}{1-p_m} - 1}{p_m}\right\}}\right.\times\\
		& \left.\qquad\qquad \left(\frac{u_m ^{(1-p_m)^{-1}-1}}{1-p_m} - j_m \left\{\frac{\frac{u_m^{(1-p_m)^{-1}-1}}{1-p_m} - 1}{p_m}\right\}\right)\right] \diff u_m.
	\end{align*}
	Solving the integral, replacing in \eqref{eq:def-tau} and simplifying yields the desired result. 
\end{proof}

The following result establishes the link between a dependence order and a corresponding measure of multivariate association, and can be found in \cite{joe1990multivariate, nelsen1996NonparametricMeasuresMultivariate,
	schmid2010copula, gijbels2021SpecificationMultivariateAssociation}.
\begin{theorem}
	Fix some vector of $\boldsymbol{p}$, then $\boldsymbol{U}\preceq_{cL}\boldsymbol{U}'$ implies $\rho^{cL}(\boldsymbol{U})\leq \rho^{cL}(\boldsymbol{U}')$, that $\boldsymbol{U}\preceq_{cU}\boldsymbol{U}'$ implies $\rho^{cU}(\boldsymbol{U})\leq \rho^{cU}(\boldsymbol{U}')$ that $\boldsymbol{U}\preceq_{c}\boldsymbol{U}'$ implies $\rho^{c}(\boldsymbol{U})\leq \rho^{c}(\boldsymbol{U}')$ and $\tau(\boldsymbol{U})\leq \tau(\boldsymbol{U}')$.
\end{theorem}

The expressions of the measures of multivariate association in Proposition \ref{prop:measures-association} can become computationally tedious in increasing dimensions. That is, if $f_{\boldsymbol{I}}(\boldsymbol{i})$ is non-zero for every $\boldsymbol{i}\in \{0, 1\}^d$, then computing $\rho^{cL}$, $\rho^{cU}$ and $\rho^{c}$ requires summing over $2^d$ terms, while computing $\tau$ requires summing over $2^{2d}$ terms. Fortunately, many distributions of interest have non-zero masses for a limited number of outcomes, thereby accelerating computations. One such example is the EPD copula when $p_1 = \dots = p_d = p$, whose expression we provided in \eqref{eq:copula-epd}. Define $\mathcal{C}^{GFGM(p)}$ as the subfamily in $\mathcal{C}^{GFGM}$ where $p_1 = \dots = p_d = p$. In this case, we find the following expressions.
\begin{proposition}\label{prop:association-measure-max}
The maximal value of $\rho^{cL}$, $\rho^{cU}$, $\rho^{c}$ and $\tau$ for copulas within $\mathcal{C}^{GFGM(p)}$ is 
$$\rho^{cL}(\boldsymbol{U}^{EPD}) = \frac{d + 1}{2^d - d - 1} \left(\left(\frac{2}{2-p}\right)^d \left[(1-p)(1-p)^{d} + p\left(\frac{3 - 2p}{2}\right)^d \right] - 1\right);$$
$$\rho ^{cU}(\boldsymbol{U}^{EPD}) =  \frac{d + 1}{2^d - d - 1} \left(\left(\frac{2}{2-p}\right)^d\left(1 - p + \frac{p}{2^d}\right) -1\right);$$
$$\rho^c(\boldsymbol{U}^{EPD}) = \frac{d + 1}{2^d - d - 1} \left( \left(\frac{2}{2-p}\right)^d\frac{1}{2}\left\{(1-p)^{d+1}+p\left(\frac{3-2p}{2}\right)^d+1-p+\frac{p}{2^d}\right\} -1\right);$$
$$\tau(\boldsymbol{U}^{EPD}) = \frac{p(1-p)}{2^{d - 1} - 1} \frac{(3-p)^d - 2(2-p)^d + (1-p)^d}{(2-p)^d}.$$
\end{proposition}

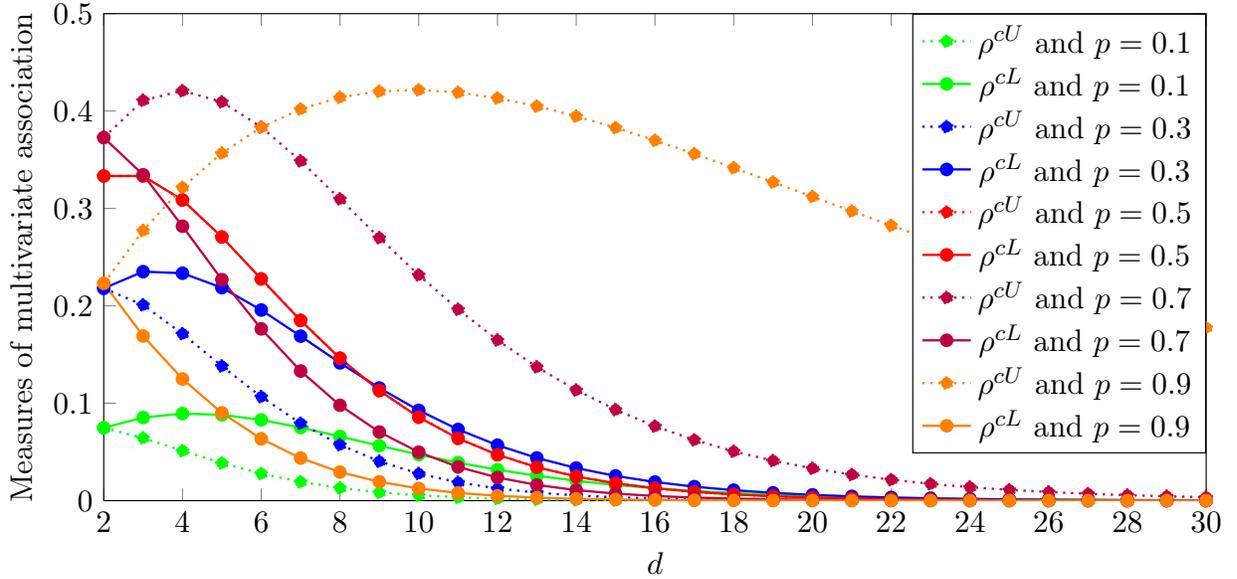
\begin{figure}[ht]
	\centering
\resizebox{\textwidth}{!}{
		\begin{tikzpicture}
	\begin{axis}[
		width = 6in, 
		height = 3in,
		ymin = 0,
		xmin = 2, 
		ymax = 0.5, 
		xmax = 30,
		xlabel={$d$},
		ylabel={Measures of multivariate association}, 
		legend style={at={(1,1)},anchor=north east},
		]
		\addplot[green, dotted, thick,mark=*] table [x="d", y="p1", col sep=comma] {code/rhoU.csv};
		\addlegendentry{$\rho^{cU}$ and $p=0.1$}	
		\addplot[green, thick,mark=*] table [x="d", y="p1", col sep=comma] {code/rhoL.csv};
		\addlegendentry{$\rho^{cL}$ and $p=0.1$}	
		
		\addplot[blue, dotted, thick,mark=*] table [x="d", y="p3", col sep=comma] {code/rhoU.csv};
		\addlegendentry{$\rho^{cU}$ and $p=0.3$}	
		\addplot[blue, thick,mark=*] table [x="d", y="p3", col sep=comma] {code/rhoL.csv};
		\addlegendentry{$\rho^{cL}$ and $p=0.3$}	
		
		\addplot[red, dotted, thick,mark=*] table [x="d", y="p5", col sep=comma] {code/rhoU.csv};
		\addlegendentry{$\rho^{cU}$ and $p=0.5$}	
		\addplot[red, thick,mark=*] table [x="d", y="p5", col sep=comma] {code/rhoL.csv};
		\addlegendentry{$\rho^{cL}$ and $p=0.5$}	
		
		\addplot[purple, dotted, thick,mark=*] table [x="d", y="p7", col sep=comma] {code/rhoU.csv};
		\addlegendentry{$\rho^{cU}$ and $p=0.7$}	
		\addplot[purple, thick,mark=*] table [x="d", y="p7", col sep=comma] {code/rhoL.csv};
		\addlegendentry{$\rho^{cL}$ and $p=0.7$}	
		
		\addplot[orange, dotted, thick,mark=*] table [x="d", y="p9", col sep=comma] {code/rhoU.csv};
		\addlegendentry{$\rho^{cU}$ and $p=0.9$}	
		\addplot[orange, thick,mark=*] table [x="d", y="p9", col sep=comma] {code/rhoL.csv};
		\addlegendentry{$\rho^{cL}$ and $p=0.9$}			
	\end{axis}
\end{tikzpicture}}
	\caption{Maximal values of $\rho^{cL}$ and $\rho^{cU}$ for a $d$-variate copula $C \in \mathcal{C}^{GFGM(p)}$ for different values of $p$ and $d$. }\label{fig:rhocl-rhocu}
\end{figure}
We now explore the effect of the shape parameter $p$ and of the dimension $d$ on the measures of association provided in Proposition \ref{prop:association-measure-max}. Figure \ref{fig:rhocl-rhocu} presents the values of $\rho^{cL}$ and $\rho^{cU}$ for different values of $p$ and of $d$. Tables \ref{tab:rhocl}, \ref{tab:rhocU}, \ref{tab:rhoc} and \ref{tab:tau} in Appendix \ref{sec:app-association} present respectively the values of $\rho^{cL}$, $\rho^{cU}$, $\rho^{c}$ and $\tau$. As noted in \cite{nelsen1996NonparametricMeasuresMultivariate}, we have $\rho^{cL} = \rho^{cU} = \rho^{c}$ when $d = 2$. Further, for a given value of $p$, $\rho^{c}$ are equal for $d = 2$ and for $d = 3$; this is also the case for $\tau$. 
\begin{remark}\label{rem:rho}
	We consider three cases:
	\begin{enumerate}
		\item For $p < 0.5$, we have $\rho^{cL} > \rho^{cU}$. Further, $\rho^{cU}$ is a strictly decreasing function of $d$, therefore its maximal value occurs for $d = 2$. We also have that $\rho^{cL}$ first increases and then decreases.
		\item For $p = 0.5$, we have $\rho^{cL} = \rho^{cU} = \rho^{c}$. Further, their values are equal for $d = 2$ and $d = 3$, then is a decreasing function of $d$. 
		\item For $p > 0.5$, we have $\rho^{cL} < \rho^{cU}$. Further, $\rho^{cL}$ is a strictly decreasing function of $d$, therefore its maximal value occurs for $d = 2$. We also have that $\rho^{cU}$ first increases and then decreases.
		\item For $d = 2$, $\rho^{cL}$ and $\rho^{cU}$ are not a monotonic function of $p$. Recall that $\boldsymbol{p}$ is a shape parameter and one should not compare the dependence structure between two copulas with different $\boldsymbol{p}$. 
	\end{enumerate}
\end{remark}

\section{Bivariate case}\label{sec:bivariate}

We study the properties of the bivariate copula from the representation in Subsection \ref{ss:representation}. Before getting started, let us recall some notions on bivariate Bernoulli distributions.

The tetrahedron in Figure \ref{fig:tetrahedron-equal} represents the convex hull of possible pmfs for bivariate Bernoulli distributions. Since a vector of pmf for a bivariate Bernoulli distribution contains four values (including one constraint that the sum of the probabilities sum to 1), one may represent the pmf with a vector of three values $(f_{00}, f_{01}, f_{10})$. Each vertex in Figure \ref{fig:tetrahedron-equal} corresponds to points where the pmf is 1 for one element and 0 for all others. One may then represent the pmf of any bivariate Bernoulli distribution as a convex combination of the four vertices. The green surface corresponds to the pmfs where $p_1 = p_2 = p$, with $p\geq 1/2$, while the purple surface corresponds to pmfs where $p_1 = p_2 = p$ with $p\leq 1/2$. The edge between $(0, 0, 0)$ and $(1, 0, 0)$ corresponds to pmfs of comonotonic bivariate Bernoulli distributions. 
\begin{figure}
	\centering	
	\resizebox{0.5\textwidth}{!}{
	\begin{tikzpicture}
		
		\newdimen\R
		\R=9pt
		
		\draw[thick,-stealth] (0,0,0)--(0,0,\R+1); 
		\draw[thick,-stealth] (0,0,0)--(0,\R+1,0);
		\draw[thick,-stealth] (0,0,0)--(\R+1,0,0);
		
		\coordinate[label=left:{$(0,0,0)$}] (O) at (0,0,0);
		\coordinate[label=left:{$(0,0,1)$}] (z) at (0,\R,0);
		\coordinate[label=left:{$(0,1,0)$}] (y) at (0,0,\R);
		\coordinate[label=below right:{$(1,0,0)$}] (x) at (\R,0,0);
		
		\node at (O) {\textbullet};
		\node at (z) {\textbullet};
		\node at (x) {\textbullet};
		\node at (y) {\textbullet};
		
		\coordinate[label={below:{$\left(\frac{1}{2},0,0\right)$}}] (como2) at (0.5*\R,0,0);
		\coordinate[label={left:{$\left(0,\frac{1}{2},\frac{1}{2}\right)$}}] (anti2) at (0,0.5*\R,0.5*\R);

		\draw[fill=gray,opacity=0.3] (x)--(y)--(z);
		
		\draw[thick] (y)--(x)--(z);
		\draw[thick, red] (y)--(z);
		\draw[thick, blue] (0,0,0)--(\R,0,0);
		
		% Pour n = 2
		
		\draw[very thick, dash dot] (como2) to (anti2);
		\node[blue] at (como2) {\LARGE\textbullet};
		\node[red] at (anti2) {\LARGE\textbullet};
		
		% Pour n >= 2
		\foreach \n in {2.8, 4, 8} {
			\coordinate (comon) at (\R/\n,0,0);
			\coordinate (antin) at (0, \R/\n, \R/\n);
			\node[blue] at (comon) {\LARGE\textbullet};
			\node[red] at (antin) {\LARGE\textbullet};
			\draw[very thick, dashed] (comon) to (antin);				
		};
		
		% Pour 1 < n < 2
		\foreach \n in {1.1, 1.3, 1.6} {
			\coordinate (comon) at (\R/\n,0,0);
			\coordinate (antin) at (2*\R/\n-\R, \R-\R/\n, \R-\R/\n);
			\node[blue] at (comon) {\textbullet};
			\node[red] at (antin) {\textbullet};
			\draw[very thick, dotted] (comon) to (antin);				
		};
		
		\draw[fill=green,opacity=0.3] (0, 0, 0)--(como2)--(anti2);
		\draw[fill=purple,opacity=0.3] (como2)--(anti2)--(\R, 0, 0);
		
\end{tikzpicture}}
	\caption{Tetrahedron representing admissible bivariate Bernoulli pmfs. The green (purple) triangle corresponds to  admissible bivariate Bernoulli pmfs for equal marginals with $p > 1/2$ ($p < 1/2$).}
	\label{fig:tetrahedron-equal}
\end{figure}
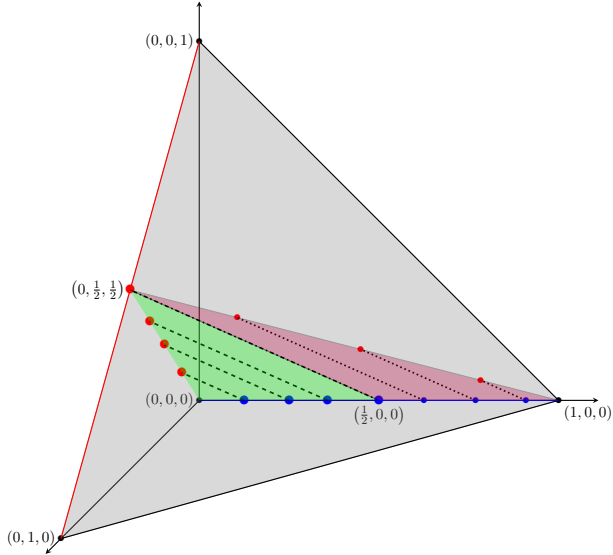

We note that in Figure \ref{fig:tetrahedron-equal}, the case $p_1 = p_2 = 0.5$ (the bivariate FGM copula) forms a line, the case $p_1 = p_2$ forms a triangle, while the case $p_1 \neq p_2$ forms a tetrahedron. It follows that each generalization enables much more flexibility in the model. 

When working in two dimensions, it is more convenient to work with an algebraic construction of a copula instead of working with pmfs. Therefore, we start by constructing an algebraic form of the pmf for bivariate Bernoulli distributions. Let $(I_1, I_2)$ be a Bernoulli random vector with success probabilities $p_1$ and $p_2$. Since the marginals are fixed, the pmf has only one free parameter. Further, as shown in \cite{fontana2018RepresentationMultivariateBernoulli}, the pmf of multivariate Bernoulli distributions with fixed margins can be expressed as a convex combination of extremal points. In the bivariate case, there are two extremal points, and these points are known analytically and correspond to the Fréchet Hoeffding upper and lower bounds. It is therefore sufficient to consider a linear function of the free parameter to construct a linear form of the pmf. We describe the bivariate pmf of $(I_1, I_2)$ with the following elements:
\begin{equation}\label{eq:pmf-bivariate-symmetric}
	\left|
	\begin{array}{cc}
		f_{00} & f_{01} \\
		f_{10} & f_{11}
	\end{array}
	\right| = 
	\left|
	\begin{array}{cc}
		(1-p_1)(1-p_2) + p_1p_2 \theta & (1-p_1)p_2 - p_1 p_2 \theta \\
		p_1(1-p_2) - p_1p_2 \theta     & p_1p_2 + p_1p_2\theta
	\end{array}
	\right|,
\end{equation}
for
\begin{equation}\label{eq:range-theta-biv}
	-\min\left(1, \frac{(1-p_1)(1-p_2)}{p_1p_2}\right) \leq \theta \leq \min\left(\frac{1-p_1}{p_1}, \frac{1-p_2}{p_2}\right).
\end{equation}
One may verify that $(I_1, I_2)$ forms a pair of counter-monotonic, independent and comonotonic rvs if $\theta$ respectively takes the values $-\min\left(1, \frac{(1-p_1)(1-p_2)}{p_1p_2}\right)$, $0$ and $\min\left(\frac{1-p_1}{p_1}, \frac{1-p_2}{p_2}\right)$. One can therefore interpret $\theta$ as a genuine dependence parameter. Further, notice that one may write
$$\theta = E\left[\frac{(I_1 - p_1)(I_2 - p_2)}{p_1p_2}\right].$$

\begin{proposition}
	The bivariate copula from \eqref{eq:copula-stochastic} has an equivalent expression as 
	\begin{equation}\label{eq:copula-biv-symmetric}
		C(u, v) = uv\left(1 + \theta \left(1 - u^{\frac{p_1}{1-p_1}}\right)\left(1 - v^{\frac{p_2}{1-p_2}}\right)\right)
	\end{equation}
	for $(u, v)\in [0, 1]^2$ and for $\theta$ satisfying \eqref{eq:range-theta-biv}.
\end{proposition}
\begin{proof}
	Inserting \eqref{eq:pmf-bivariate-symmetric} into \eqref{eq:copula-stochastic} and simplifying (a lot) yields the desired result.	
\end{proof}
When $p_1 = p_2 = p$, the copula in \eqref{eq:copula-biv-symmetric} corresponds to a known extension of the FGM copula, proposed in Section 2 of \cite{huang1999modifications}, whose expression is 
\begin{equation}\label{eq:huang-kotz}
	C(u, v) = uv \left(1 + a \left(1-u^b\right)\left(1-v^b\right)\right),
\end{equation}
for $(u, v)\in [0, 1]^d$, $b > 0$ and $-(\max(1, b))^{-2} \leq a \leq b^{-1}$. It turns out that another contribution of this paper is to identify the stochastic representation of the Huang-Kotz FGM copula in \eqref{eq:huang-kotz}. Further, the Huang-Kotz FGM copula is symmetric, meaning that $C(u, v) = C(v, u)$ for all $(u, v) \in [0, 1]^2$. When $p_1 \neq p_2$, the copula is no longer symmetric and we obtain a copula that \cite{bairamov2002dependence} calls an asymmetric modified Huang-Kotz FGM distribution. Much of \cite{huang1999modifications} is dedicated to determining the feasible range for the parameter $a$, that is, verifying that the $2$-increasing property in \eqref{eq:2-increasing} is satisfied. It is much easier to verify that the underlying pmf for the bivariate Bernoulli random vector is admissible since this involves verifying that the pmf lies between two extremal points. Also, it becomes natural to extend the Huang-Kotz copula to higher dimensions by constructing a $d$-variate random vector $\boldsymbol{I}$ with equal probabilities, that is, with $p_1 = \dots = p_d$. 

We present the heatmaps of the copula density function in Figure \ref{fig:max-dep} for the case with maximal dependence, that is, when $(I_1, I_2) = (I_1^c, I_2^c)$. One notices that for $(p_1, p_2) \in (0.5, 1)^2$, there is more mass in the upper tail, while for $(p_1, p_2) \in (0, 0.5)^2$, there is more mass in the lower tail. This is in line with the observations from the numerical examples of the maximal association measures in Section \ref{ss:association}. 
\begin{figure}[ht]
	\centering
	\begin{subfigure}[b]{0.3\textwidth}
		\centering
		\includegraphics[width=\textwidth]{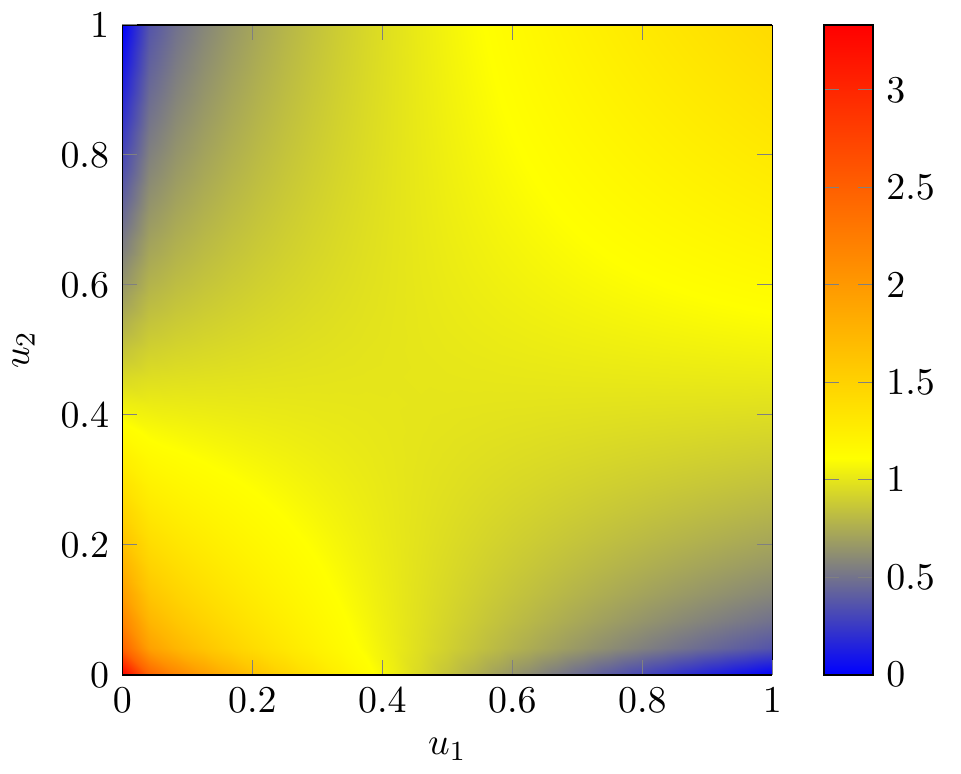}
		\caption{$p_1 = 0.3$, $p_2 = 0.3$}
		\label{fig:33p}
	\end{subfigure}
	\hfill
	\begin{subfigure}[b]{0.3\textwidth}
		\centering
		\includegraphics[width=\textwidth]{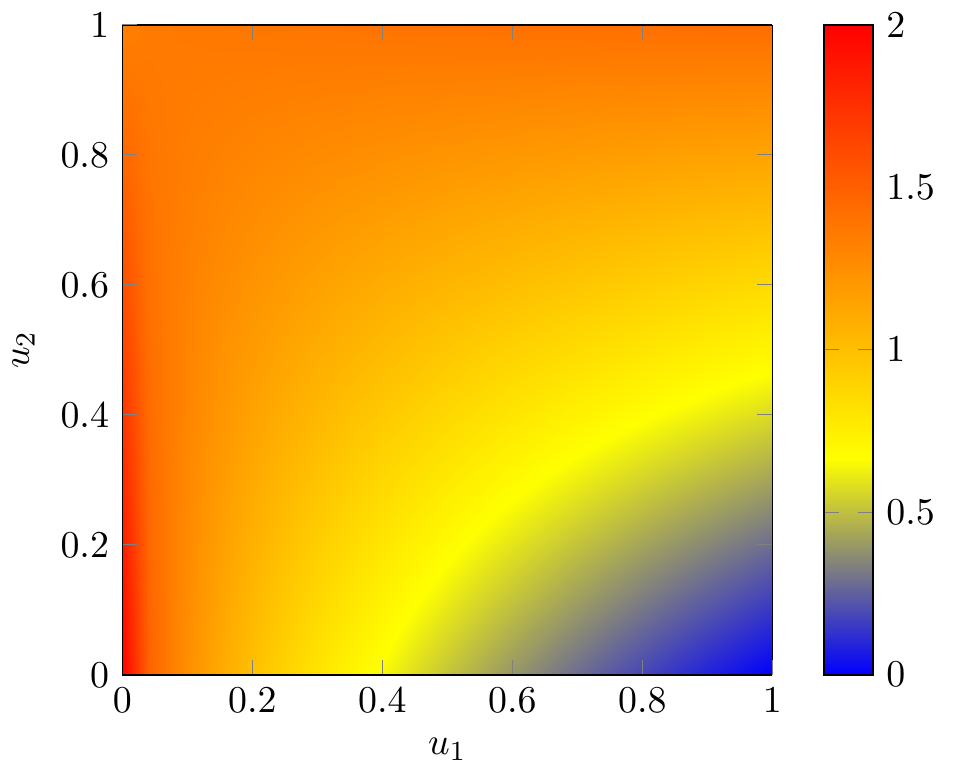}
		\caption{$p_1 = 0.3$, $p_2 = 0.5$}
		\label{fig:35p}
	\end{subfigure}
	\hfill
	\begin{subfigure}[b]{0.3\textwidth}
		\centering
		\includegraphics[width=\textwidth]{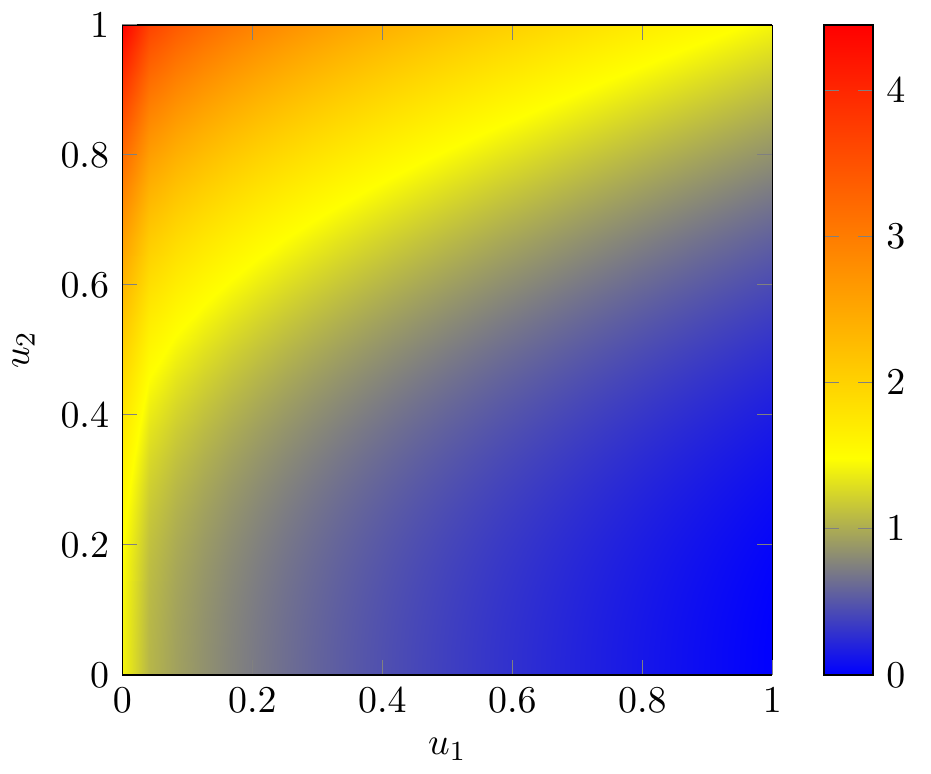}
		\caption{$p_1 = 0.3$, $p_2 = 0.7$}
		\label{fig:37p}
	\end{subfigure} \\
	
	\begin{subfigure}[b]{0.3\textwidth}
		\centering
		\includegraphics[width=\textwidth]{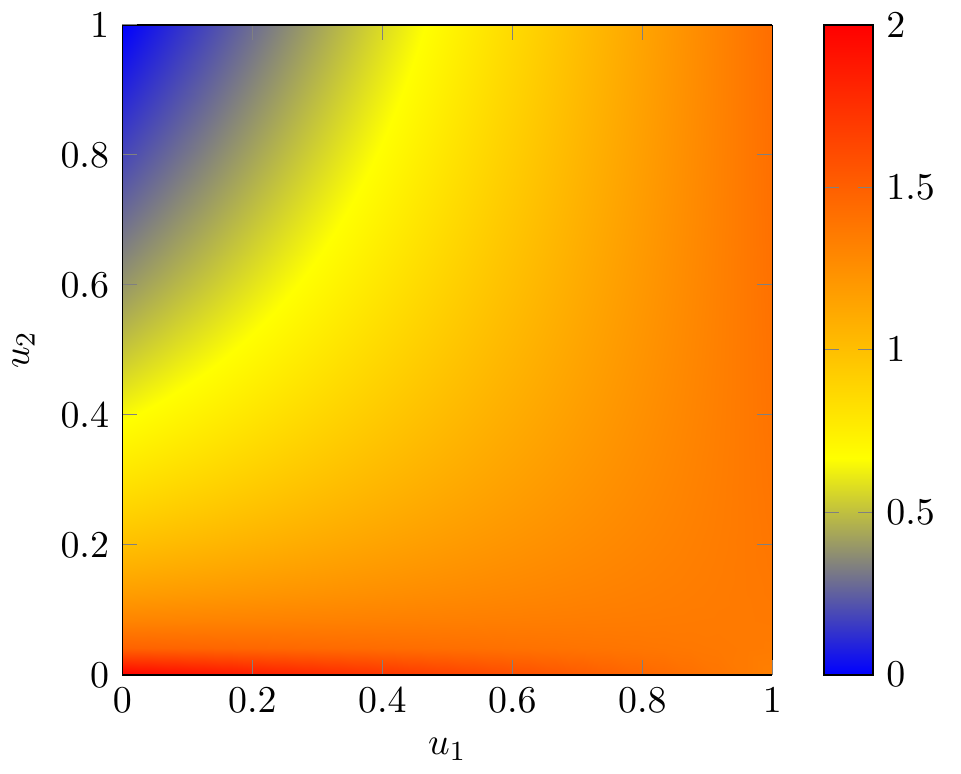}
		\caption{$p_1 = 0.5$, $p_2 = 0.3$}
		\label{fig:53p}
	\end{subfigure}
	\hfill
	\begin{subfigure}[b]{0.3\textwidth}
		\centering
		\includegraphics[width=\textwidth]{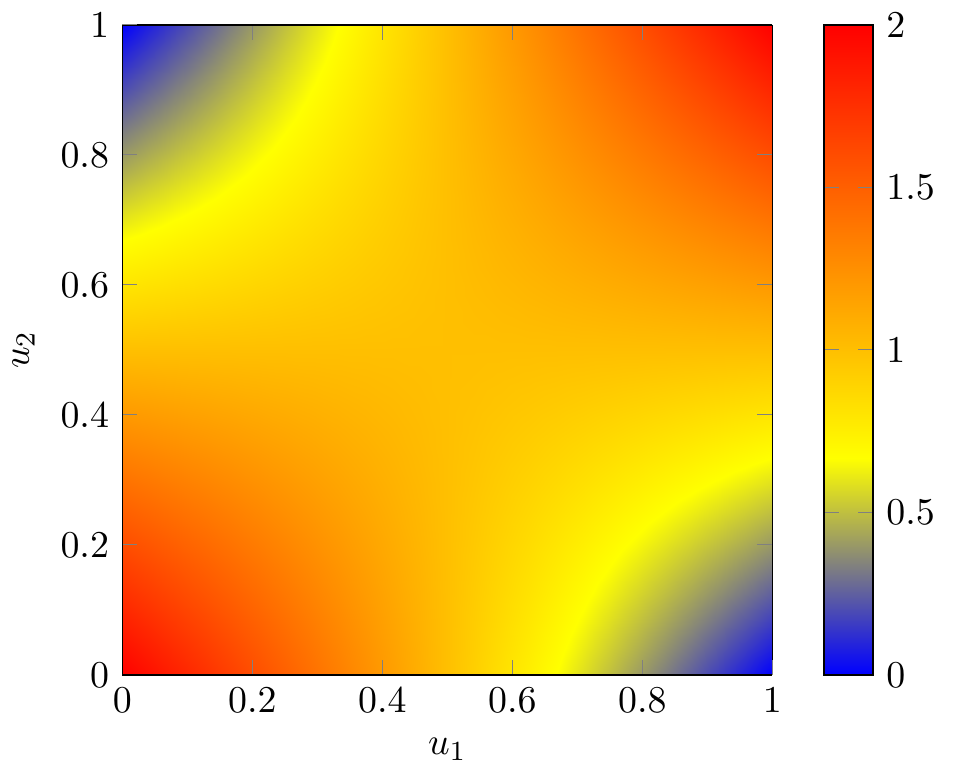}
		\caption{$p_1 = 0.5$, $p_2 = 0.5$}
		\label{fig:55p}
	\end{subfigure}
	\hfill
	\begin{subfigure}[b]{0.3\textwidth}
		\centering
		\includegraphics[width=\textwidth]{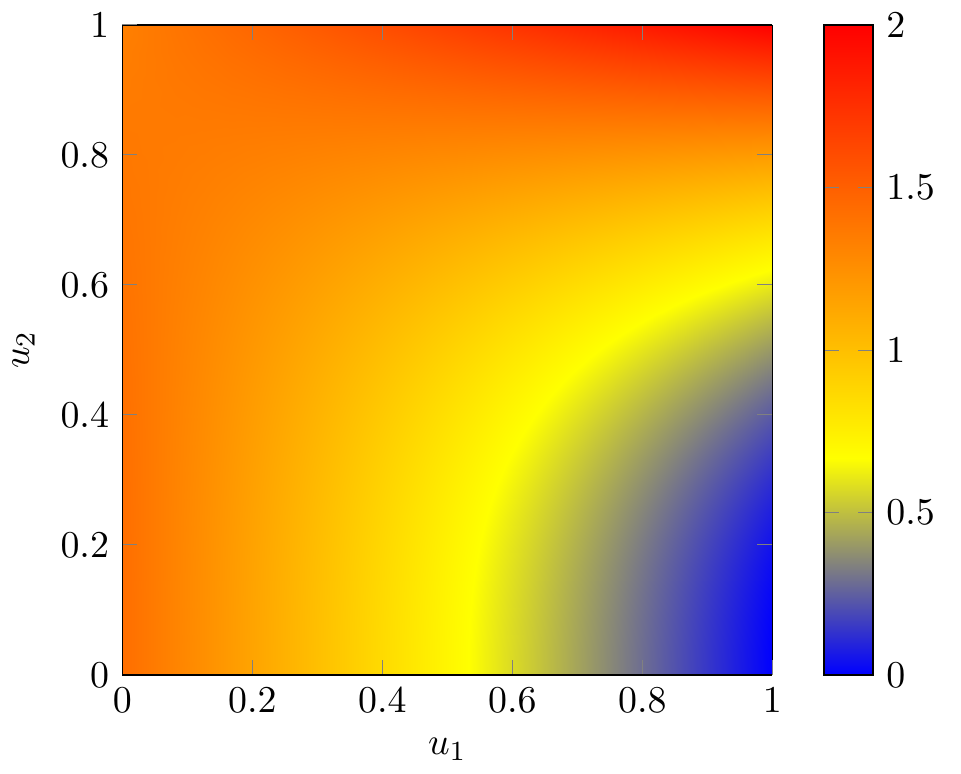}
		\caption{$p_1 = 0.5$, $p_2 = 0.7$}
		\label{fig:57p}
	\end{subfigure} \\
	
	\begin{subfigure}[b]{0.3\textwidth}
		\centering
		\includegraphics[width=\textwidth]{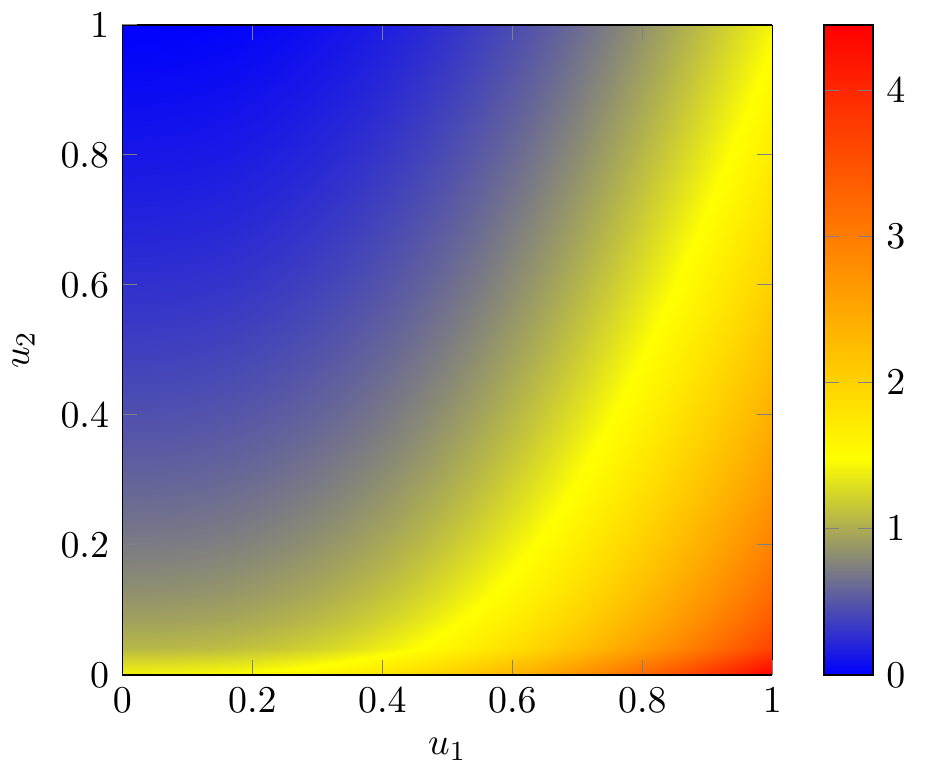}
		\caption{$p_1 = 0.7$, $p_2 = 0.3$}
		\label{fig:73p}
	\end{subfigure}
	\hfill
	\begin{subfigure}[b]{0.3\textwidth}
		\centering
		\includegraphics[width=\textwidth]{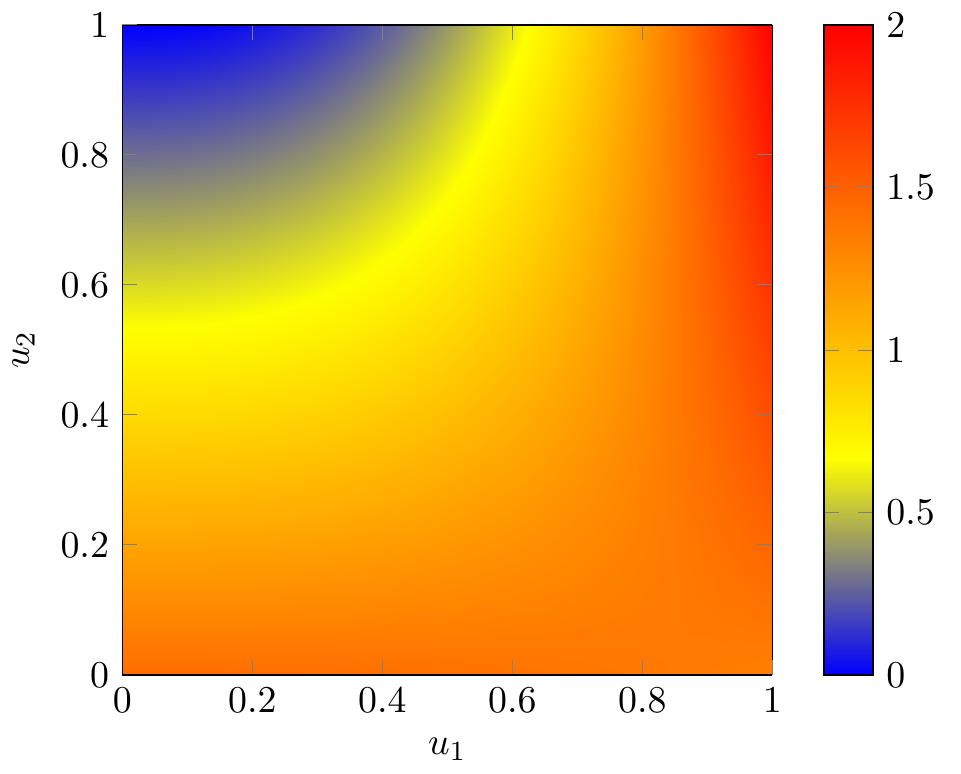}
		\caption{$p_1 = 0.7$, $p_2 = 0.5$}
		\label{fig:75p}
	\end{subfigure}
	\hfill
	\begin{subfigure}[b]{0.3\textwidth}
		\centering
		\includegraphics[width=\textwidth]{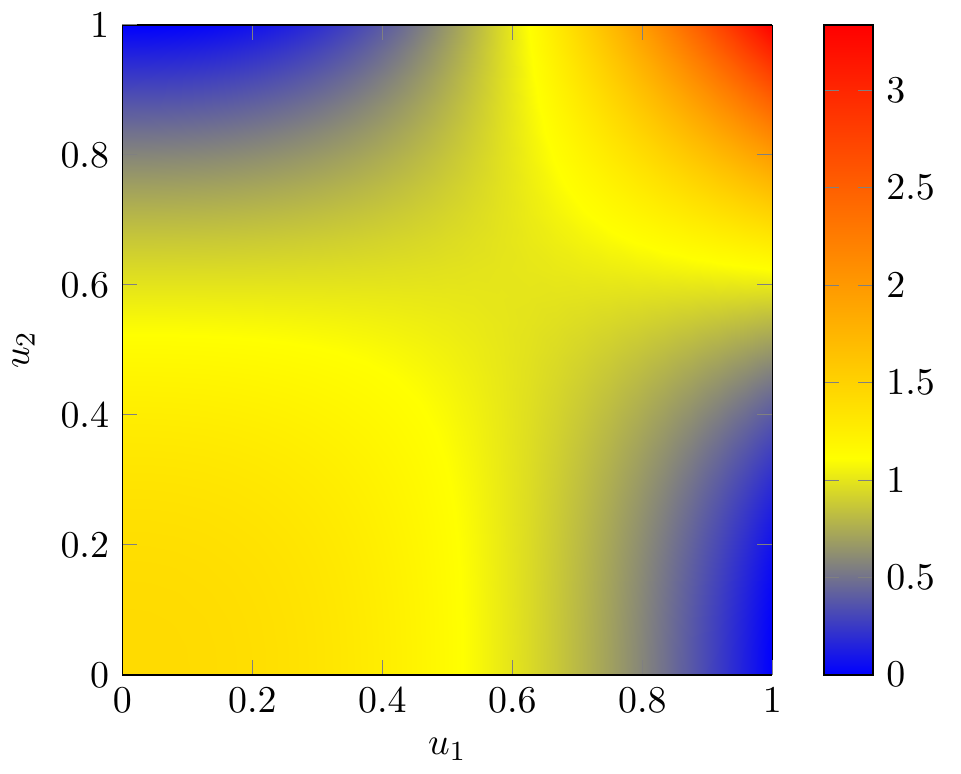}
		\caption{$p_1 = 0.7$, $p_2 = 0.7$}
		\label{fig:77p}
	\end{subfigure} \\
	\caption{Heatmaps for density functions associated to the maximal dependence structure.}
	\label{fig:max-dep}
\end{figure}
Figure \ref{fig:min-dep} presents the density function when $(I_1, I_2)$ forms a pair of counter-monotonic rvs (with cdf constructed by the Fréchet lower bound). One observes that for negative dependence, the mass is mostly uniform on the unit cube; this is a consequence of moderate dependence from the FGM copula. However, when $p \leq 0.5$, there is less mass near the coordinate $(1, 1)$, while for $p \geq 0.5$, there is less mass near the coordinate $(0, 0)$. 
\begin{figure}[ht]
	\centering
	\begin{subfigure}[b]{0.3\textwidth}
		\centering
		\includegraphics[width=\textwidth]{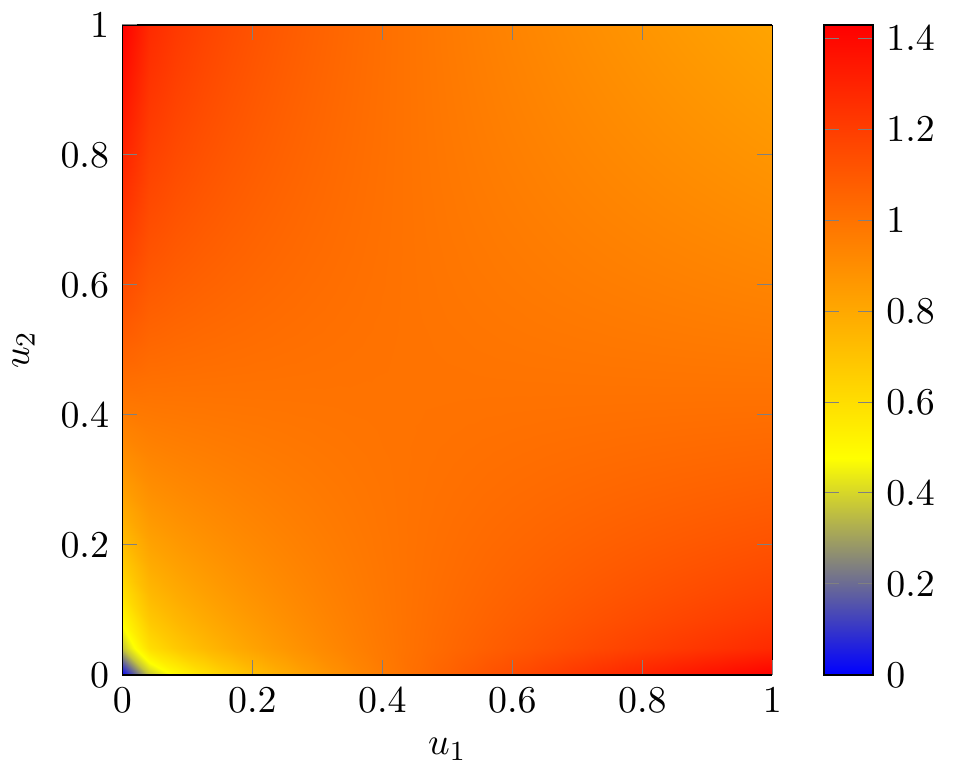}
		\caption{$p_1 = 0.3$, $p_2 = 0.3$}
		\label{fig:33m}
	\end{subfigure}
	\hfill
	\begin{subfigure}[b]{0.3\textwidth}
		\centering
		\includegraphics[width=\textwidth]{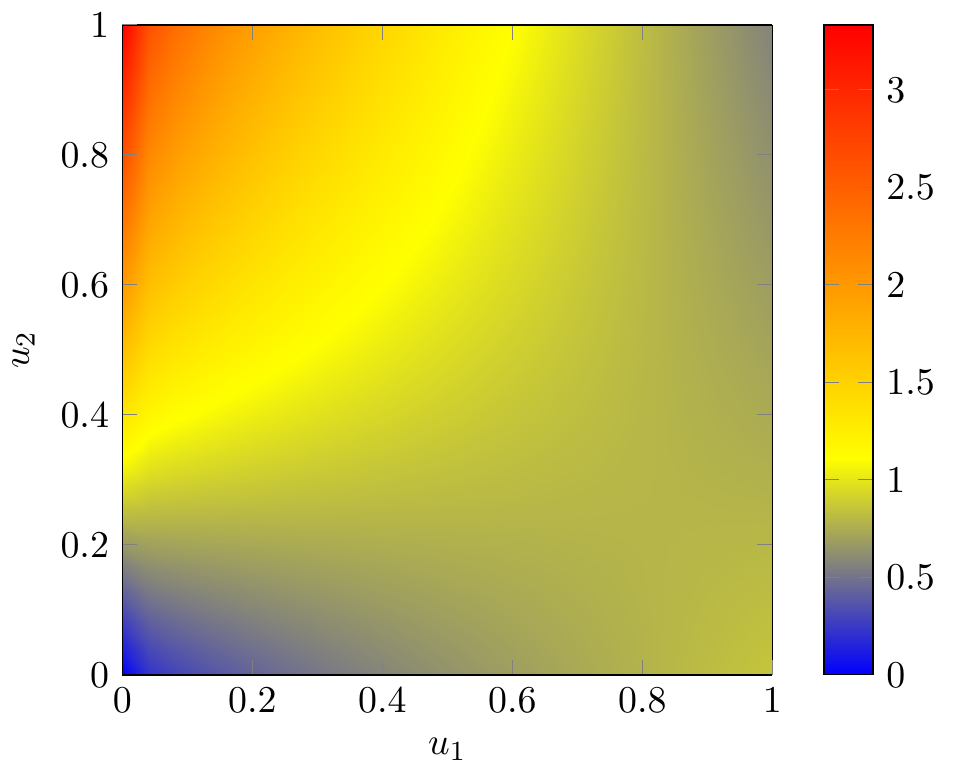}
		\caption{$p_1 = 0.3$, $p_2 = 0.5$}
		\label{fig:35m}
	\end{subfigure}
	\hfill
	\begin{subfigure}[b]{0.3\textwidth}
		\centering
		\includegraphics[width=\textwidth]{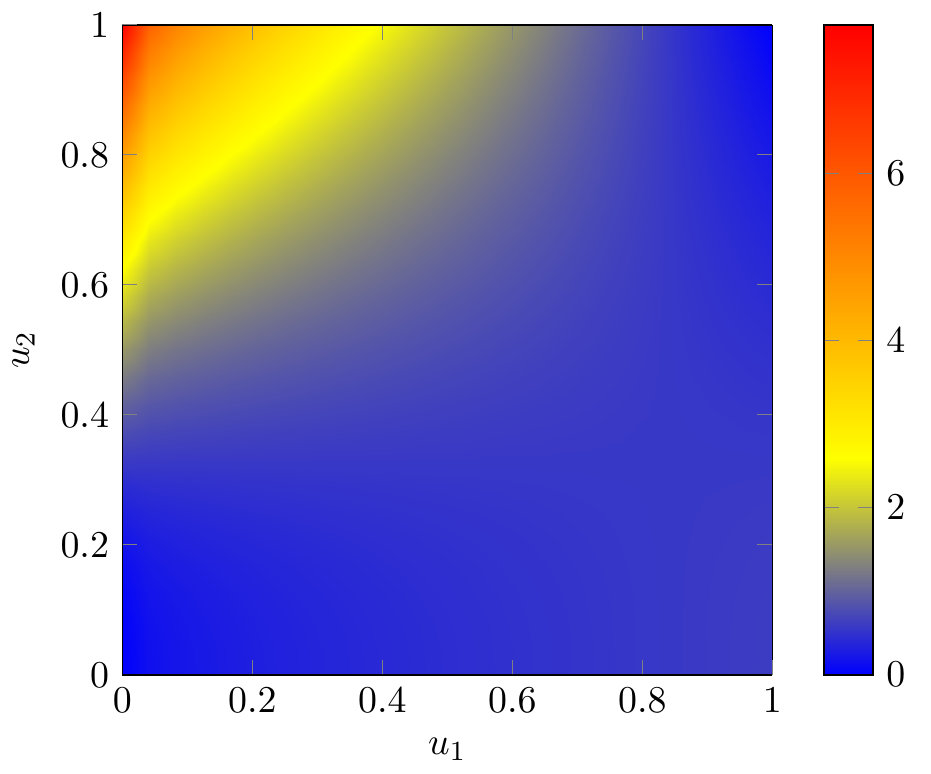}
		\caption{$p_1 = 0.3$, $p_2 = 0.7$}
		\label{fig:37m}
	\end{subfigure} \\
	
	\begin{subfigure}[b]{0.3\textwidth}
		\centering
		\includegraphics[width=\textwidth]{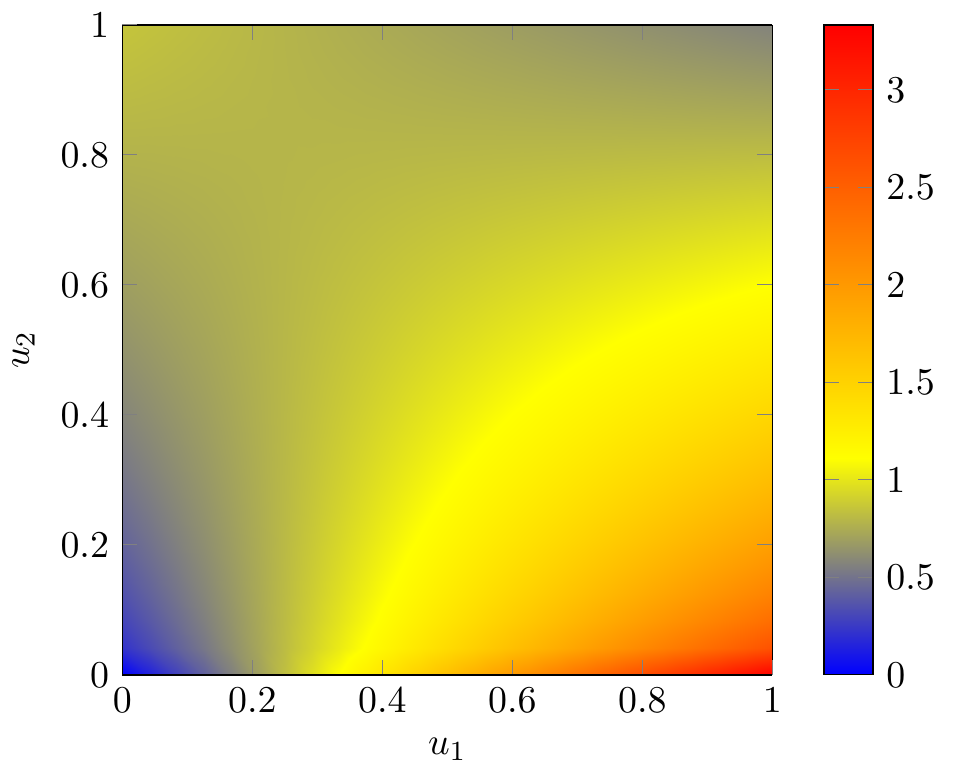}
		\caption{$p_1 = 0.5$, $p_2 = 0.3$}
		\label{fig:53m}
	\end{subfigure}
	\hfill
	\begin{subfigure}[b]{0.3\textwidth}
		\centering
		\includegraphics[width=\textwidth]{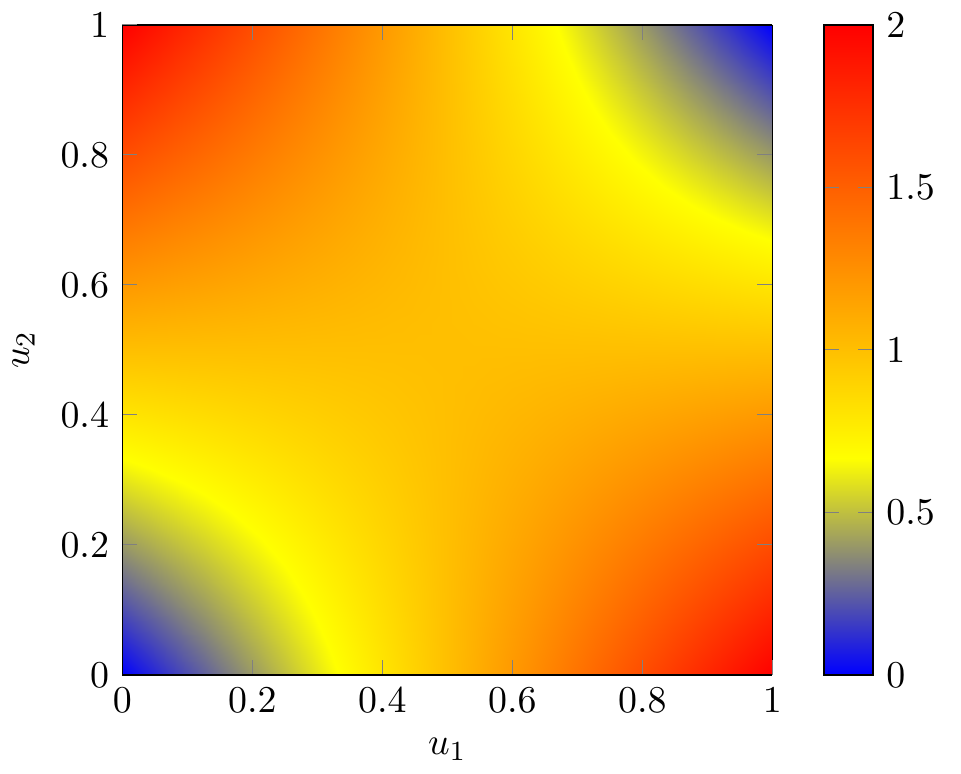}
		\caption{$p_1 = 0.5$, $p_2 = 0.5$}
		\label{fig:55m}
	\end{subfigure}
	\hfill
	\begin{subfigure}[b]{0.3\textwidth}
		\centering
		\includegraphics[width=\textwidth]{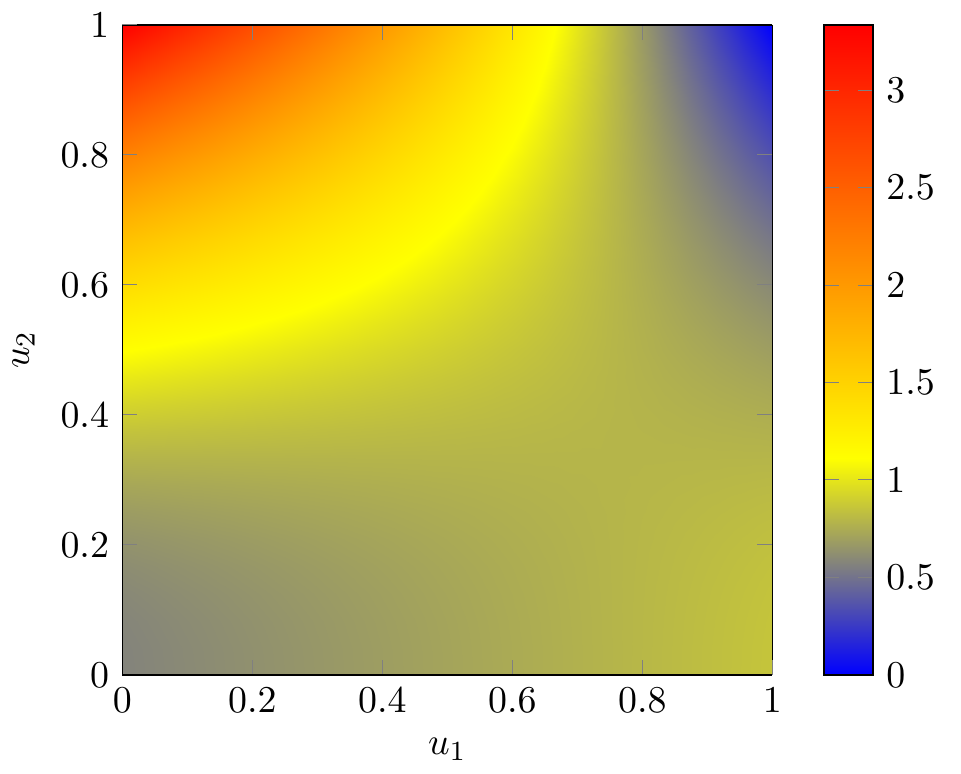}
		\caption{$p_1 = 0.5$, $p_2 = 0.7$}
		\label{fig:57m}
	\end{subfigure} \\
	
	\begin{subfigure}[b]{0.3\textwidth}
		\centering
		\includegraphics[width=\textwidth]{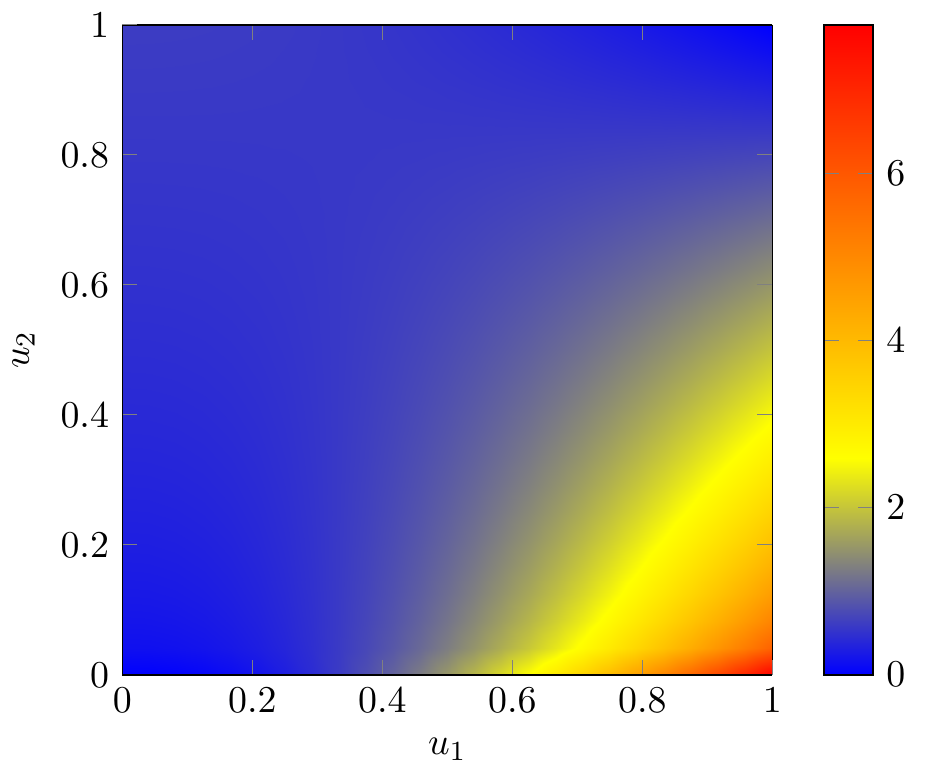}
		\caption{$p_1 = 0.7$, $p_2 = 0.3$}
		\label{fig:73m}
	\end{subfigure}
	\hfill
	\begin{subfigure}[b]{0.3\textwidth}
		\centering
		\includegraphics[width=\textwidth]{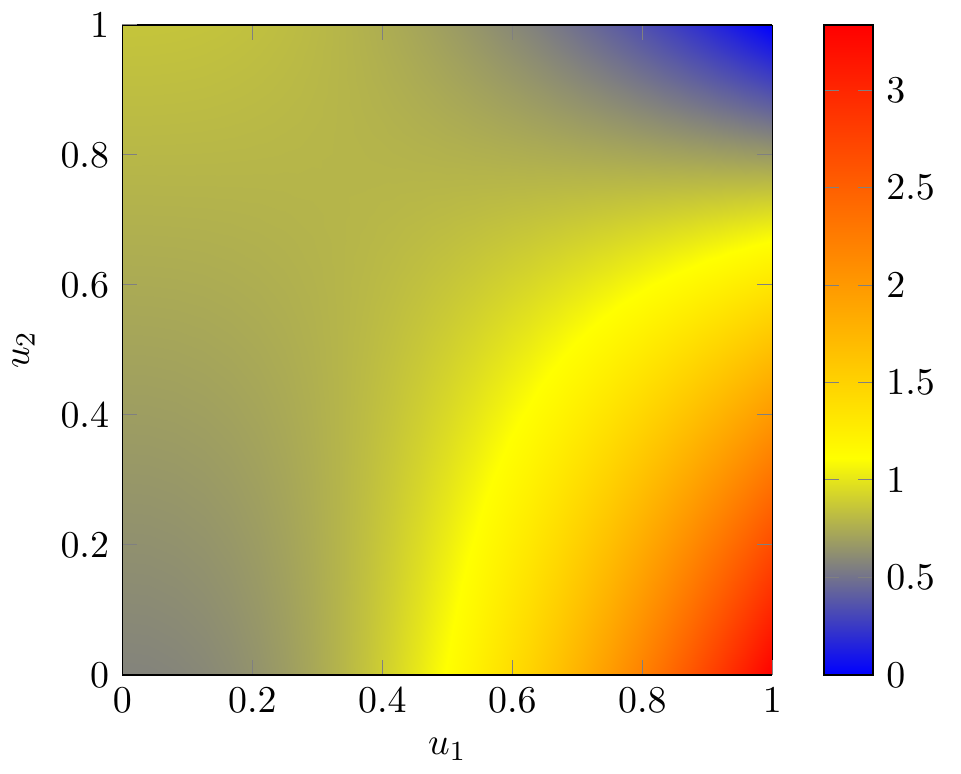}
		\caption{$p_1 = 0.7$, $p_2 = 0.5$}
		\label{fig:75m}
	\end{subfigure}
	\hfill
	\begin{subfigure}[b]{0.3\textwidth}
		\centering
		\includegraphics[width=\textwidth]{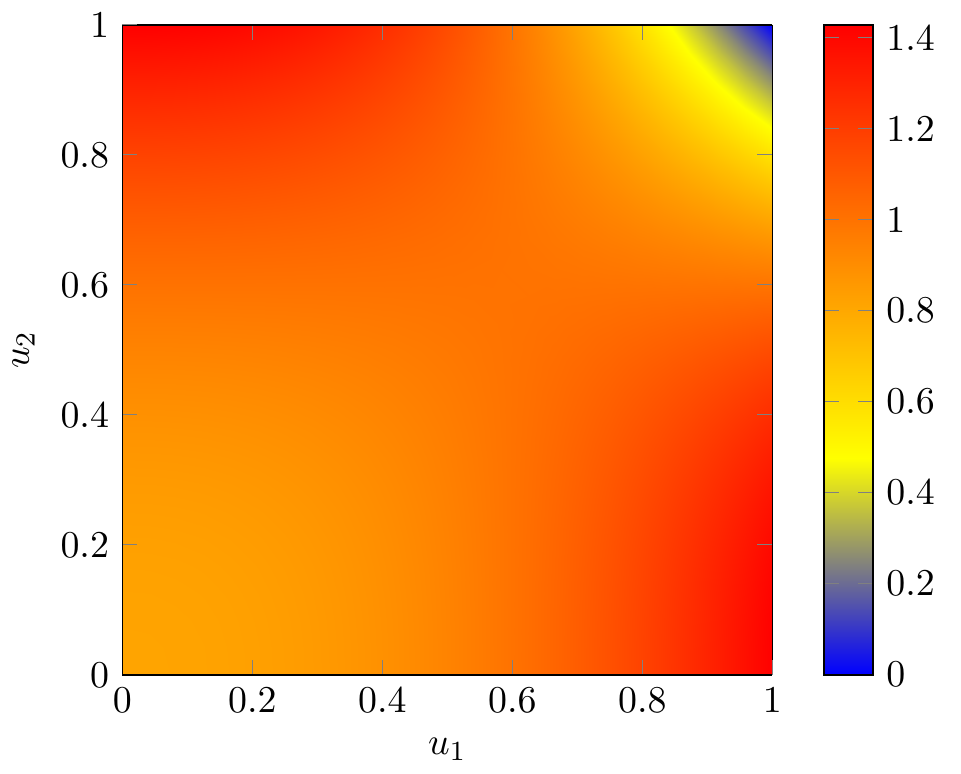}
		\caption{$p_1 = 0.7$, $p_2 = 0.7$}
		\label{fig:77m}
	\end{subfigure} \\
	\caption{Heatmaps for density functions associated to the minimal dependence structure.}
	\label{fig:min-dep}
\end{figure}

\section{Family of exchangeable GFGM copulas}\label{sec:subfamilies}

Up to this point, we have investigated a new family of copulas based on a stochastic representation of random vectors with uniform margins. We have discussed the effect of the parameter $\boldsymbol{p}$ on the shape of dependence. However, we have not investigated how to construct pmfs for $\boldsymbol{I}$ with given margins, unless in the bivariate case. 

Since the class of $d$-variate Bernoulli distributions with given margins is large, it will be more practical for model construction and for applications to consider subfamilies of multivariate Bernoulli distributions in order to construct copulas. Many subfamilies of multivariate Bernoulli distributions are useful within high-dimensional modelling (see, e.g., Sections 7.1 and 8.1.2 of \cite{joe1997MultivariateModelsMultivariate} for some examples, or \cite{jiang2021set} for a set of dependence structures which lead to efficient high-dimensional simulation procedures). In this section, we explore some properties of the subclass of $\mathcal{C}^{GFGM}$ that corresponds to exchangeable dependence structures; we will denote this subclass as $\mathcal{C}^{eGFGM}.$

\subsection{Extremal points}

Fix some probability $p \in (0, 1)$ and let $N_d$ be a rv defined as $I_1 + \dots + I_d$. Further, define $\mathcal{N}_d$ as the class of univariate pmfs for rvs with support $\{0, 1, \dots, d\}$ and mean $pd$. The authors of \cite{fontana2021ModelRiskCredit} prove a one-to-one correspondence between the class of pmfs for exchangeable Bernoulli random vectors with fixed mean $p$ and $\mathcal{N}_d$. As explained in \cite{blier-wong2022ExchangeableFGMCopulas}, extremal points of pmfs for symmetric multivariate Bernoulli distributions map to extremal points for FGM copulas. This property also holds for copulas in $\mathcal{C}^{eGFGM}$. It follows for a given value of $p$ and $d$, any copula in $\mathcal{C}^{eGFGM}$ can be expressed as a convex combination of extremal copulas of $\mathcal{C}^{eGFGM}$. Further, the following proposition, adapted from  \cite{fontana2021ModelRiskCredit}, provides an explicit method to obtain the extremal copulas.
\begin{proposition}
	The pmfs of the rvs $N_{d}$ corresponding to the extremal points of $\mathcal{N}_d$ are 
	$$\Pr(N_{d} = k) = \begin{cases}
		\frac{j_2 - pd}{j_2 - j_2},& \quad k = j_1,\\
		\frac{pd - j_1}{j_2 - j_2},& \quad k = j_2,\\
		0, & \text{Otherwise}
	\end{cases},$$
	with $j_1 \in \{0, 1, \dots, j_1^{\wedge}\}$ and $j_2 \in \{j_2^{\vee}, j_2^{\vee} + 1, \dots, d\}$, where $j_1^{\wedge} < pd < j_2^{\vee}$ If $pd$ is integer, the extreme points also contain the degenerate rv at $k = pd$. 
\end{proposition}

\subsection{Extreme negative dependence}

Within the subfamily of exchangeable copulas constructed from exchangeable multivariate Bernoulli and Coxian-2 distributions, it is possible to identify the dependence structure which leads to the lower bound under the supermodular order. Let $\boldsymbol{I}^{-} = (I^{-}_1, \dots, I^{-}_d)$ be a random vector with pmf 
$$\Pr(I_1^{-} = i_1, \dots, I_d^{-} = i_d) = \begin{cases}
(j_2^{\vee} - dp) \binom{d}{j_1^{\wedge}}^{-1}, & \text{ if } \sum_{j = 1}^d i_j = j_1^{\wedge} \text{ and } pd \text{ is not integer}\\
(dp - j_1^{\wedge}) \binom{d}{j_2^{\vee}}^{-1}, & \text{ if } \sum_{j = 1}^d i_j = j_2^{\vee} \text{ and } pd \text{ is not integer}\\
\binom{d}{pd}^{-1}, & \text{ if } \sum_{j = 1}^d i_j = pd \text{ and } pd \text{ is integer}\\
0, & \text{Otherwise}
\end{cases}.$$
Then, for any exchangeable Bernoulli random vector $\boldsymbol{I}$, it holds that $\boldsymbol{I}^{-}\preceq_{sm} \boldsymbol{I} \preceq_{sm} \boldsymbol{I}^{c}$. It follows from Theorem \ref{thm:stochastic-orders} that the exchangeable random vector with representation as in \eqref{eq:representation-u} corresponding to the lower bound under the supermodular order, denoted as $\boldsymbol{U}^{END}$, is $\boldsymbol{U}^{END} = \boldsymbol{U}_0^{1-p}\boldsymbol{U}_1^{I^{-}}.$
\begin{proposition}
	Let $\boldsymbol{U}$ be an exchangeable random vector with representation as in \eqref{eq:representation-u}. If $pd$ is not an integer, then the minimal values of $\rho^{cL}$ and $\rho^{cU}$ are 
	$$\rho^{cL}\left(\boldsymbol{U}^{END}\right) = \frac{d - 1}{2^d -d-1} \left[\left(\frac{2(1-p)}{2-p}\right)^{d}\left(\frac{3-2p}{2-2p}\right)^{j_1^{\wedge}}\left(j + 1 - pd + \frac{(3-2p)(pd-j)}{2-2p}\right) - 1\right]$$
	and
	$$\rho^{cU} \left(\boldsymbol{U}^{END}\right) = \frac{d - 1}{2^d - d - 1} \left[\frac{2^{d-j_1^{\wedge}}}{(2-p)^{d}}\left(\frac{j_1^{\wedge} - pd}{2} + 1\right) - 1\right].$$
	If $pd$ is an integer, then the minimal values of $\rho^{cL}$ and $\rho^{cU}$ are 
	$$\rho^{cL}\left(\boldsymbol{U}^{END}\right) = \frac{d - 1}{2^d -d-1} \left[\left(\frac{(3-2p)^{p}(2-2p)^{1-p}}{2-p}\right)^d - 1\right]$$
	and
	$$\rho^{cU} \left(\boldsymbol{U}^{END}\right) = \frac{d - 1}{2^d - d - 1} \left[\left(\frac{2^{1-p}}{2-p}\right)^d - 1\right].$$
\end{proposition}

In Tables \ref{tab:rhoclmin} and \ref{tab:rhocumin} of Appendix \ref{sec:app-association}, we provide the values of $\rho^{cL}$ and $\rho^{cU}$ for the random vector $\boldsymbol{U}^{END}$, with different values of $p$ and $d$.
%As in Remark \ref{rem:rho}, this numerical comparison is helpful to understand the difference between the two measures of association. From Figure \ref{fig:min-dep}, we observe for $p_1 = p_2 = 0.3$ that there is less mass in the lower tail

\subsection{Mixture construction}

We close this section by proposing a simple method to construct high-dimensional exchangeable copulas based on Bernoulli and Coxian-2 distributions. Let $\Lambda$ be a rv with support on the unit interval. We define the pmf of $\boldsymbol{I}$ using the mixture construction
\begin{equation}\label{eq:finite-mixture}
	f_{\boldsymbol{I}}(\boldsymbol{i}) = \int_{0}^{1} \lambda^{i_\bullet}(1-\lambda)^{d - i_\bullet} \diff F_\Lambda(\lambda),
\end{equation}
where $i_\bullet = i_1 + \dots + i_d$, for $\boldsymbol{i} \in \{0, 1\}^d$. The classical result from \cite{de1929funzione} states that if $\boldsymbol{I}$ is sampled from an infinite sequence of exchangeable Bernoulli rvs, there exists a rv $\Lambda$ such that \eqref{eq:finite-mixture} holds. Within this construction, the value of $p$ is given by $E[\Lambda]$. 

\begin{proposition}
	Let $\Lambda$ be a rv with support in $[0, 1]$ and let $p = E[\Lambda]$. Let $\boldsymbol{I}$ be a random vector with pmf as defined in \eqref{eq:finite-mixture}. Further, let  $\boldsymbol{U}$ be a random vector with a stochastic representation $\boldsymbol{U}_0^{1-p}\boldsymbol{U}_1^{\boldsymbol{I}}$. Then, the copula associated with the cdf of $\boldsymbol{U}$ is 
	\begin{align*}
		C(\boldsymbol{u}) &= E\left[\prod_{m = 1}^{d} \left(u_m^{(1-p)^{-1}} - \frac{\Lambda}{p} \left\{u_m^{(1-p)^{-1}} - u\right\} \right)\right]\\
		&= \prod_{m = 1}^d u_m^{(1-p)^{-1}} \left(1 + \sum_{k = 1}^d (-1)^k\sum_{1 \leq j_1 < \dots < j_k \leq d} \frac{E[\Lambda^k]}{p^k} \left(1 - u_{j_1}^{-\frac{p}{1-p}}\right) \dots \left(1 - u_{j_k}^{-\frac{p}{1-p}}\right)\right),
	\end{align*}
	for $\boldsymbol{u} \in [0, 1]^d$.
\end{proposition}
From the previous proposition, we present a subfamily of eGFGM copulas constructed with beta distributions.
\begin{example}
	Let $\Lambda \sim Beta(\alpha, \beta)$, for $\alpha > 0$ and $\beta > 0$. Then, we have $p = \alpha/(\alpha + \beta)$ and the copula becomes
	$$C(\boldsymbol{u}) = \prod_{m = 1}^d u_m\left(1 + \sum_{k = 1}^d \sum_{1 \leq j_1 < \dots < j_k \leq d} E\left[\left(\frac{(\alpha + \beta)\Lambda - \alpha}{\alpha}\right)^k\right] \left(1 + \frac{\alpha}{\beta}\right)^k \left(1 - u_{j_1}^{\frac{\alpha}{\beta}}\right) \dots \left(1 - u_{j_k}^{\frac{\alpha}{\beta}}\right)\right).$$
\end{example}

\section{Conclusion}\label{sec:conclusion}

% Figure \ref{fig:tetrahedron-equal} provides a convenient summary of the copula in two dimensions. Vanilla FGM copulas form a line between coordinates $(1/2, 0, 0)$ and $(0, 1/2, 1/2)$. Extensions in Section \ref{ss:symmetric} form a triangle between coordinates $(0, 0, 0)$, $(0, 1/2, 1/2)$ and $(1, 0, 0)$. Finally, the asymmetric extension of Section \ref{ss:asymmetric} fills the Fréchet class of bivariate Bernoulli random vectors. 

In this paper, we explore a new family of copulas constructed with multivariate Bernoulli and Coxian-2 distributions. This family extends FGM copulas, enabling stronger dependence and asymmetry between margins. This family of copulas is constructed with a stochastic representation, such that high-dimensional generalizations are simple. 

In two dimensions, we recover one of the Huang-Kotz FGM copulas when $p_1 = p_2$, and one of the modified asymmetric Huang-Kotz FGM copulas when $p_1 \neq p_2$. Therefore, another contribution of this paper is to propose a stochastic representation of (modified asymmetric) Huang-Kotz FGM copulas. Further, the class $\mathcal{C}^{GFGM}$ naturally extends the (modified asymmetric) Huang-Kotz FGM copulas. In the same spirit, it will be useful to find a stochastic representation of the second family of Huang-Kotz copulas, whose expression is
$$C(u, v) = uv \left(1 + a \left(1-u\right)^b\left(1-v\right)^b\right),$$
for $(u, v) \in [0, 1]^d$, for $-1 \leq a \leq [(p+1)/(p-1)]^{b-1}$ and $p>1$.

Coxian-2 distributions are special cases of phase-type distributions with two phases. Just as Bernstein copulas extend FGM copulas, further research involves extending the copula proposed in this paper to more general phase-type distributions. 

\section{Acknowledgments}

This work was partially supported by the Natural Sciences and Engineering Research Council of Canada (Blier-Wong: 559169, Cossette: 04273; Marceau: 05605). The first author worked on this paper while on a visit to CREST, ENSAE Paris. 

\appendix

\section{Values of multivarite association measures}\label{sec:app-association}

\begin{table}[ht]
	\centering\begin{tabular}{rrrrrrrrrr}
		\diagbox[width=2em]{$p$}{$d$} &      2 &      3 &      5 &      8 &     10 &     15 &     20 &     50 &    100 \\
		0.1 & 0.0748 & 0.0853 & 0.0881 & 0.0659 & 0.0473 & 0.0161 & 0.0047 & 0.0000 & 0.0000 \\
		0.2 & 0.1481 & 0.1646 & 0.1619 & 0.1130 & 0.0777 & 0.0239 & 0.0062 & 0.0000 & 0.0000 \\
		0.3 & 0.2180 & 0.2351 & 0.2187 & 0.1414 & 0.0927 & 0.0254 & 0.0059 & 0.0000 & 0.0000 \\
		0.4 & 0.2812 & 0.2930 & 0.2558 & 0.1520 & 0.0944 & 0.0227 & 0.0047 & 0.0000 & 0.0000 \\
		0.5 & 0.3333 & 0.3333 & 0.2707 & 0.1463 & 0.0856 & 0.0178 & 0.0031 & 0.0000 & 0.0000 \\
		0.6 & 0.3673 & 0.3499 & 0.2613 & 0.1270 & 0.0696 & 0.0122 & 0.0018 & 0.0000 & 0.0000 \\
		0.7 & 0.3728 & 0.3345 & 0.2269 & 0.0979 & 0.0498 & 0.0072 & 0.0009 & 0.0000 & 0.0000 \\
		0.8 & 0.3333 & 0.2778 & 0.1684 & 0.0636 & 0.0297 & 0.0035 & 0.0003 & 0.0000 & 0.0000 \\
		0.9 & 0.2231 & 0.1690 & 0.0901 & 0.0293 & 0.0125 & 0.0011 & 0.0001 & 0.0000 & 0.0000
	\end{tabular}
	\caption{Maximal values of $\rho^{cL}$ for $d$-variate copulas $C \in \mathcal{C}^{GFGM(p)}$ for different values of $p$ and $d$.}\label{tab:rhocl}
\end{table}

\begin{table}[ht]
	\centering\begin{tabular}{rrrrrrrrrr}
		\diagbox[width=2em]{$p$}{$d$} &      2 &      3 &      5 &      8 &     10 &     15 &     20 &     50 &    100 \\
		0.1 & 0.0748 & 0.0643 & 0.0386 & 0.0130 & 0.0055 & 0.0005 & 0.0000 & 0.0000 & 0.0000 \\
		0.2 & 0.1481 & 0.1317 & 0.0843 & 0.0313 & 0.0141 & 0.0014 & 0.0001 & 0.0000 & 0.0000 \\
		0.3 & 0.2180 & 0.2009 & 0.1382 & 0.0573 & 0.0278 & 0.0034 & 0.0003 & 0.0000 & 0.0000 \\
		0.4 & 0.2812 & 0.2695 & 0.2006 & 0.0942 & 0.0499 & 0.0078 & 0.0010 & 0.0000 & 0.0000 \\
		0.5 & 0.3333 & 0.3333 & 0.2707 & 0.1463 & 0.0856 & 0.0178 & 0.0031 & 0.0000 & 0.0000 \\
		0.6 & 0.3673 & 0.3848 & 0.3442 & 0.2179 & 0.1431 & 0.0407 & 0.0100 & 0.0000 & 0.0000 \\
		0.7 & 0.3728 & 0.4110 & 0.4094 & 0.3097 & 0.2317 & 0.0933 & 0.0331 & 0.0000 & 0.0000 \\
		0.8 & 0.3333 & 0.3889 & 0.4370 & 0.4042 & 0.3497 & 0.2073 & 0.1095 & 0.0011 & 0.0000 \\
		0.9 & 0.2231 & 0.2772 & 0.3567 & 0.4140 & 0.4216 & 0.3828 & 0.3121 & 0.0434 & 0.0007
	\end{tabular}
	\caption{Maximal values of $\rho^{cU}$ for $d$-variate copulas $C \in \mathcal{C}^{GFGM(p)}$ for different values of $p$ and $d$.}\label{tab:rhocU}
\end{table}

\begin{table}[ht]
	\centering\begin{tabular}{rrrrrrrrrr}
		\diagbox[width=2em]{$p$}{$d$} &      2 &      3 &      5 &      8 &     10 &     15 &     20 &     50 &    100 \\
		0.1 & 0.0748 & 0.0748 & 0.0633 & 0.0395 & 0.0264 & 0.0083 & 0.0023 & 0.0000 & 0.0000 \\
		0.2 & 0.1481 & 0.1481 & 0.1231 & 0.0722 & 0.0459 & 0.0126 & 0.0032 & 0.0000 & 0.0000 \\
		0.3 & 0.2180 & 0.2180 & 0.1784 & 0.0994 & 0.0602 & 0.0144 & 0.0031 & 0.0000 & 0.0000 \\
		0.4 & 0.2812 & 0.2812 & 0.2282 & 0.1231 & 0.0721 & 0.0153 & 0.0028 & 0.0000 & 0.0000 \\
		0.5 & 0.3333 & 0.3333 & 0.2707 & 0.1463 & 0.0856 & 0.0178 & 0.0031 & 0.0000 & 0.0000 \\
		0.6 & 0.3673 & 0.3673 & 0.3028 & 0.1724 & 0.1064 & 0.0264 & 0.0059 & 0.0000 & 0.0000 \\
		0.7 & 0.3728 & 0.3728 & 0.3181 & 0.2038 & 0.1407 & 0.0503 & 0.0170 & 0.0000 & 0.0000 \\
		0.8 & 0.3333 & 0.3333 & 0.3027 & 0.2339 & 0.1897 & 0.1054 & 0.0549 & 0.0006 & 0.0000 \\
		0.9 & 0.2231 & 0.2231 & 0.2234 & 0.2217 & 0.2170 & 0.1920 & 0.1561 & 0.0217 & 0.0004
	\end{tabular}
	\caption{Maximal values of $\rho^{c}$for a $d$-variate copula $C \in \mathcal{C}^{GFGM(p)}$ for different values of $p$ and $d$.}\label{tab:rhoc}
\end{table}

\begin{table}[ht]
	\centering\begin{tabular}{rrrrrrrrrr}
		\diagbox[width=2em]{$p$}{$d$} &      2 &      3 &      5 &      8 &     10 &     15 &     20 &     50 &    100 \\
		0.1 & 0.0499 & 0.0499 & 0.0378 & 0.0195 & 0.0117 & 0.0031 & 0.0008 & 0.0000 & 0.0000 \\
		0.2 & 0.0988 & 0.0988 & 0.0760 & 0.0407 & 0.0253 & 0.0074 & 0.0021 & 0.0000 & 0.0000 \\
		0.3 & 0.1453 & 0.1453 & 0.1136 & 0.0636 & 0.0411 & 0.0132 & 0.0042 & 0.0000 & 0.0000 \\
		0.4 & 0.1875 & 0.1875 & 0.1494 & 0.0881 & 0.0594 & 0.0213 & 0.0075 & 0.0000 & 0.0000 \\
		0.5 & 0.2222 & 0.2222 & 0.1811 & 0.1133 & 0.0799 & 0.0324 & 0.0130 & 0.0001 & 0.0000 \\
		0.6 & 0.2449 & 0.2449 & 0.2049 & 0.1372 & 0.1020 & 0.0475 & 0.0220 & 0.0002 & 0.0000 \\
		0.7 & 0.2485 & 0.2485 & 0.2147 & 0.1554 & 0.1227 & 0.0667 & 0.0362 & 0.0009 & 0.0000 \\
		0.8 & 0.2222 & 0.2222 & 0.1996 & 0.1583 & 0.1337 & 0.0867 & 0.0562 & 0.0041 & 0.0001 \\
		0.9 & 0.1488 & 0.1488 & 0.1402 & 0.1236 & 0.1129 & 0.0896 & 0.0710 & 0.0176 & 0.0017
	\end{tabular}
	\caption{Maximal values of $\tau$ for a $d$-variate copula $C \in \mathcal{C}^{GFGM(p)}$ for different values of $p$ and $d$.}\label{tab:tau}
\end{table}

\begin{table}[ht]
	\centering
	\begin{tabular}{rrrrrrrr}
		\diagbox[width=2em]{$p$}{$d$} &       2 &       3 &       4 &       5 &       8 &      10 &      15 \\
		                          0.1 & -0.0083 & -0.0080 & -0.0070 & -0.0057 & -0.0023 & -0.0010 & -0.0001 \\
		                          0.2 & -0.0370 & -0.0343 & -0.0289 & -0.0227 & -0.0047 & -0.0020 & -0.0001 \\
		                          0.3 & -0.0934 & -0.0824 & -0.0449 & -0.0274 & -0.0073 & -0.0030 & -0.0002 \\
		                          0.4 & -0.1875 & -0.0977 & -0.0590 & -0.0467 & -0.0102 & -0.0040 & -0.0002 \\
		                          0.5 & -0.3333 & -0.1111 & -0.0954 & -0.0484 & -0.0137 & -0.0048 & -0.0003 \\
		                          0.6 & -0.2449 & -0.1603 & -0.0865 & -0.0706 & -0.0152 & -0.0056 & -0.0003 \\
		                          0.7 & -0.1598 & -0.1843 & -0.1123 & -0.0627 & -0.0162 & -0.0063 & -0.0003 \\
		                          0.8 & -0.0833 & -0.0926 & -0.0936 & -0.0883 & -0.0160 & -0.0067 & -0.0004 \\
		                          0.9 & -0.0248 & -0.0263 & -0.0254 & -0.0228 & -0.0121 & -0.0065 & -0.0003 \\
	\end{tabular}
	\caption{Minimal values of $\rho^{cL}$ for different values of $p$ and $d$ for copulas in $\mathcal{C}^{eGFGM}$.}\label{tab:rhoclmin}
\end{table}

%\begin{table}[ht]
%	\centering
%	\begin{tabular}{rrrrrrr}
%		\diagbox[width=2em]{$p$}{$d$} &       2 &       3 &       5 &       8 &      10 &      15 \\ 
%		                          0.1 & -0.0083 & -0.0080 & -0.0057 & -0.0023 & -0.0010 & -0.0001 \\
%		                          0.2 & -0.0370 & -0.0343 & -0.0227 & -0.0047 & -0.0020 & -0.0001 \\
%		                          0.3 & -0.0934 & -0.0824 & -0.0274 & -0.0073 & -0.0030 & -0.0002 \\
%		                          0.4 & -0.1875 & -0.0977 & -0.0467 & -0.0102 & -0.0040 & -0.0002 \\
%		                          0.5 & -0.3333 & -0.1111 & -0.0484 & -0.0137 & -0.0048 & -0.0003 \\
%		                          0.6 & -0.2449 & -0.1603 & -0.0706 & -0.0152 & -0.0056 & -0.0003 \\
%		                          0.7 & -0.1598 & -0.1843 & -0.0627 & -0.0162 & -0.0063 & -0.0003 \\
%		                          0.8 & -0.0833 & -0.0926 & -0.0883 & -0.0160 & -0.0067 & -0.0004 \\
%		                          0.9 & -0.0248 & -0.0263 & -0.0228 & -0.0121 & -0.0065 & -0.0003 
%	\end{tabular}
%	\caption{Minimal values of $\rho^{cL}$ for different values of $p$ and $d$ in the exchangeable case.}\label{tab:rhoclmin}
%\end{table}

\begin{table}[ht]
	\centering
	\begin{tabular}{rrrrrrrr}
		\diagbox[width=2em]{$p$}{$d$} &       2 &       3 &       4 &       5 &       8 &      10 &      15 \\
		                          0.1 & -0.0083 & -0.0086 & -0.0081 & -0.0071 & -0.0035 & -0.0018 & -0.0001 \\
		                          0.2 & -0.0370 & -0.0398 & -0.0389 & -0.0354 & -0.0068 & -0.0031 & -0.0002 \\
		                          0.3 & -0.0934 & -0.1044 & -0.0627 & -0.0357 & -0.0097 & -0.0040 & -0.0002 \\
		                          0.4 & -0.1875 & -0.1211 & -0.0661 & -0.0547 & -0.0120 & -0.0045 & -0.0003 \\
		                          0.5 & -0.3333 & -0.1111 & -0.0954 & -0.0484 & -0.0137 & -0.0048 & -0.0003 \\
		                          0.6 & -0.2449 & -0.1254 & -0.0759 & -0.0591 & -0.0127 & -0.0049 & -0.0003 \\
		                          0.7 & -0.1598 & -0.1352 & -0.0726 & -0.0443 & -0.0114 & -0.0046 & -0.0003 \\
		                          0.8 & -0.0833 & -0.0741 & -0.0600 & -0.0453 & -0.0093 & -0.0038 & -0.0002 \\
		                          0.9 & -0.0248 & -0.0233 & -0.0199 & -0.0158 & -0.0058 & -0.0025 & -0.0001 \\
	\end{tabular}
	\caption{Minimal values of $\rho^{cU}$ for different values of $p$ and $d$ for copulas in $\mathcal{C}^{eGFGM}$.}\label{tab:rhocumin}
\end{table}

%
%\begin{table}[ht]
%	\centering
%	\begin{tabular}{rrrrrrr}
%		\diagbox[width=2em]{$p$}{$d$} &       2 &       3 &       5 &       8 &      10 &      15 \\ 
%		                          0.1 & -0.0083 & -0.0086 & -0.0071 & -0.0035 & -0.0018 & -0.0001 \\
%		                          0.2 & -0.0370 & -0.0398 & -0.0354 & -0.0068 & -0.0031 & -0.0002 \\
%		                          0.3 & -0.0934 & -0.1044 & -0.0357 & -0.0097 & -0.0040 & -0.0002 \\
%		                          0.4 & -0.1875 & -0.1211 & -0.0547 & -0.0120 & -0.0045 & -0.0003 \\
%		                          0.5 & -0.3333 & -0.1111 & -0.0484 & -0.0137 & -0.0048 & -0.0003 \\
%		                          0.6 & -0.2449 & -0.1254 & -0.0591 & -0.0127 & -0.0049 & -0.0003 \\
%		                          0.7 & -0.1598 & -0.1352 & -0.0443 & -0.0114 & -0.0046 & -0.0003 \\
%		                          0.8 & -0.0833 & -0.0741 & -0.0453 & -0.0093 & -0.0038 & -0.0002 \\
%		                          0.9 & -0.0248 & -0.0233 & -0.0158 & -0.0058 & -0.0025 & -0.0001
%	\end{tabular}
%	\caption{Minimal values of $\rho^{cU}$ for different values of $p$ and $d$ in the exchangeable case.}\label{tab:rhocumin}
%\end{table}

\bibliographystyle{apalike}
\bibliography{ref}

\end{document}